\documentclass[english]{amsart}
\usepackage[T1]{fontenc}
\usepackage[latin9]{inputenc}
\usepackage{units}
\usepackage{mathrsfs}
\usepackage{mathtools}
\usepackage{amsbsy}
\usepackage{amstext}
\usepackage{amsthm}
\usepackage{amssymb}
\usepackage{stmaryrd}

\makeatletter
\numberwithin{equation}{section}
  \theoremstyle{definition}
  \newtheorem{example}{\protect\examplename}[section]
  \theoremstyle{definition}
  \newtheorem{defn}{\protect\definitionname}[section]
  \theoremstyle{plain}
  \newtheorem{lem}{\protect\lemmaname}[section]
  \theoremstyle{plain}
  \newtheorem{prop}{\protect\propositionname}[section]
  \theoremstyle{plain}
  \newtheorem{cor}{\protect\corollaryname}[section]

\usepackage{upgreek}
\usepackage{amssymb}
\let\originalleft\left
\let\originalright\right
\renewcommand{\left}{\mathopen{}\mathclose\bgroup\originalleft}
\renewcommand{\right}{\aftergroup\egroup\originalright}

\makeatother

\usepackage{babel}
  \providecommand{\definitionname}{Definition}
  \providecommand{\examplename}{Example}
  \providecommand{\lemmaname}{Lemma}
  \providecommand{\propositionname}{Proposition}
\providecommand{\corollaryname}{Corollary}

\begin{document}

\title[Magnitudes Reborn]{Magnitudes Reborn:\\
 Quantity Spaces as Scalable Monoids}

\author{Dan Jonsson}

\address{Dan Jonsson, University of Gothenburg, SE-405 30 Gothenburg, Sweden }

\email{dan.jonsson@gu.se}
\begin{abstract}
In ancient Greek mathematics, magnitudes such as lengths were strictly
distinguished from numbers. In modern quantity calculus, a distinction
is made between quantities and scalars that serve as measures of quantities.
The author believes, for reasons apparent from this article, that
quantities should play a major rather than a minor role in modern
mathematics.

The extended Introduction includes a survey of the historical development
and theoretical structure of the pre-modern theory of magnitudes and
numbers. In Part 1, work, insights and controversies related to quantity
calculus from Euler onward are reviewed. In Parts 2 and 3, we define
scalable monoids and, as a special case, quantity spaces; both can
be regarded as universal algebras. Scalable monoids are related to
rings and modules, and quantity spaces are to scalable monoids as
vector spaces are to modules. Subalgebras and homomorphic images of
scalable monoids can be formed, and tensor products of scalable monoids
can be constructed as well. We also define and investigate congruence
relations on scalable monoids, unit elements of scalable monoids,
basis-like substructures of scalable monoids and quantity spaces,
and scalar representations of elements of quantity spaces. The mathematical
theory of quantity spaces is presented with a view to metrological
applications.

This article supersedes arXiv:1503.00564 and complements arXiv:1408.5024.
\end{abstract}

\maketitle

\section*{Introduction}

Formulas such as $E=\tfrac{mv^{2}}{2}$ or $\frac{\partial T}{\partial t}=\kappa\frac{\partial^{2}T}{\partial x^{2}}$,
used to express physical laws,\linebreak{}
describe relationships between scalars, commonly real numbers. An
alternative interpretation of such equations is possible, however.
Since the scalars assigned to the variables in these equations are
numerical measures of certain quantities, the equations express relationships
between these quantities as well. For example, $E=\tfrac{mv^{2}}{2}$
can also be interpreted as describing a relation between an energy
$E$, a mass $m$ and a velocity $v$ \textendash{} three underlying
physical quantities, whose existence and properties do not depend
on the scalars that may be used to represent them. With this interpretation,
though, $\tfrac{mv^{2}}{2}$ and similar expressions will have meaning
only if operations on quantities, corresponding to operations on numbers,
are defined. In other words, an appropriate way of calculating with
quantities, a \emph{quantity calculus}, needs to be available.

In a useful survey \cite{BOER}, de Boer described the development
of quantity calculus until the late 20th century, starting with Maxwell's
\cite{MAXW} concept of a physical quantity $q$ comprised of a unit
quantity $\left[q\right]$ of the same kind as $q$ and a scalar $\left\{ q\right\} $
which is the measure of $q$ relative to $\left[q\right]$, so that
we can write $q=\left\{ q\right\} \left[q\right]$. Like Wallot, who
reformulated $q=\left\{ q\right\} \left[q\right]$ as $q=\left(q/\left[q\right]\right)\left[q\right]$
in an important article in \emph{Handbuch der Physik} 1926 \cite{WAL},
de Boer argued, though, that the notion of a physical quantity should
be a primitive one, not expressed by means of non-quantities.

The roots of quantity calculus go far deeper in the history of mathematics
than to Wallot, however, or even to Maxwell or other scientists of
the modern era, such as Fourier \cite{FOUR}; the origins of quantity
calculus can be traced back to ancient Greek geometry and arithmetic,
as codified in Euclid's \emph{Elements \cite{EUCL}}. 

Of fundamental importance in the \emph{Elements} is the distinction
between \emph{numbers (multitudes)} and \emph{magnitudes}. The notion
of a number (\emph{arithmos}) is based on that of a ''unit'' or
''monad'' (\emph{monas}); a number is ''a multitude composed of
units''. Thus, a number is essentially a positive integer. (A collection
of units containing just one unit was not, in principle, considered
to be a multitude of units in Greek arithmetic, so $1$ was not, strictly
speaking, a number.) Numbers can be compared, added and multiplied,
and a smaller number can be subtracted from a larger one, but the
ratio of two numbers $m,n$ is not itself a number but just a pair
$m:n$ expressing relative size. Ratios can, however, be compared;
$m:n=m':n'$ means that $mn'=nm'$. A bigger number $m$ is said to
be measured by a smaller number $k$ if $m=rk$ for some number $r$;
a prime number is a number that is not measured by any other number
(or measured only by $1$), and $m,n$ are relatively prime when there
is no number (except 1) measuring both. 

Magnitudes (\emph{megethos}), on the other hand, are phenomena such
as lengths, areas, volumes or times. Unlike numbers, magnitudes are
of \emph{different kinds, }and while the magnitudes of a particular
kind correspond loosely to numbers, making measurement of magnitudes
possible, the magnitudes form a continuum, and there is no distinguished
''unit magnitude''. In Greek mathematics, magnitudes of the same
kind can be compared and added, and a smaller magnitude can be subtracted
from a larger one of the same kind, but magnitudes cannot, in general,
be multiplied or divided. One can form the ratio of two magnitudes
of the same kind, $p$ and $q$, but this is not a magnitude but just
a pair $p:q$ expressing relative size. A greater magnitude $q$ is
said to be measured by a smaller magnitude $u$ if there is a number
$n$ such that $q$ is equal to $u$ taken $n$ times; we may write
this as $q=n\times u$ here.

Remarkably, the first three propositions about magnitudes proved by
Euclid in the \emph{Elements} are, in the notation used here, 
\begin{gather*}
n\times(u_{1}\dotplus\cdots\dotplus u_{k})=n\times u_{1}\dotplus\cdots\dotplus n\times u_{k},\\
(n_{1}+\cdots+n_{k})\times u=n_{1}\times u\dotplus\cdots\dotplus n_{k}\times u,\qquad m\times(n\times u)=(mn)\times u,
\end{gather*}
where $n,m,n_{1},\ldots,n_{k}$ are numbers (\emph{arithmoi}), $u$
is a magnitude, $u_{1},\ldots,u_{k}$ are magnitudes of the same kind,
and $q_{1}\dotplus\cdots\dotplus q_{k}$ is the sum of magnitudes.
As shown in Section \ref{subsec:91}, these identities are fundamental
in modern quantity calculus as well.

If $p$ and $q$ are magnitudes of the same kind, and there is some
magnitude $u$ of this kind and some numbers $m,n$ such that $p=m\times u$
and $q=n\times u$, then $p$ and $q$ are said to be ''commensurable'';
the ratio of magnitudes $p:q$ can then be represented by the ratio
of numbers $m:n$.\footnote{It is natural to assume that if $p=m\times u=m'\times u'$ and $q=n\times u=n'\times u'$
then $m:n=m':n'$, so that the representation of $p:q$ is unique.} However, magnitudes may also be ''relatively prime''; it may happen
that $p:q$ cannot be expressed as $m:n$ for any numbers $m,n$ because
there are no $m,n,u$ such that $p=m\times u$ and $q=n\times u$.
In view of the Pythagorean philosophical conviction of the primacy
of numbers, the discovery of examples of such ''incommensurable''
magnitudes created a deep crisis in early Greek mathematics \cite{HASS},
a crisis that also affected the foundations of geometry. If ratios
of \emph{arithmoi} do not always suffice to represent ratios of magnitudes,
it seems that it would not always be possible to express in terms
of \emph{arithmoi} the fact that two ratios of magnitudes are equal,
as are the ratios of the lengths of corresponding sides of similar
triangles. This difficulty was resolved by Eudoxos, who realized that
a ''proportion'', that is, a relation among magnitudes of the form
\emph{''}$p$ is to $q$ as $p'$ is to $q'$\emph{''}, conveniently
denoted $p\vcentcolon q\dblcolon p'\vcentcolon q'$, can be defined
numerically even if there is no pair of ratios of \emph{arithmoi}
$m:n$ and $m':n'$ corresponding to $p:q$ and $p':q'$, respectively,
so that $p:q\dblcolon p':q'$ cannot be inferred from $m:n=m':n'$.
Specifically, Eudoxos invented an ingenious indirect way of determining
if $p\vcentcolon q\dblcolon p'\vcentcolon q'$ in terms of nothing
but \emph{arithmoi} by means of a construction similar to the Dedekind
cut, as described in Book V of the \emph{Elements}. Using modern terminology,
one can say that Eudoxos defined an equivalence relation $\!\dblcolon$
between pairs of magnitudes of the same kind in terms of positive
integers, and as a consequence it became possible to conceptualize
in terms of \emph{arithmoi} not only ratios of magnitudes corresponding
to rational numbers but also ratios of magnitudes corresponding to
irrational numbers. Eudoxos thus reconciled the continuum of magnitudes
with the discrete \emph{arithmoi}, but in retrospect this feat reduced
the incentive to rethink the Greek notion of number, to generalize
the \emph{arithmoi}.

To summarize, Greek mathematicians used two notions of muchness, and
built a theoretical system around each notion. These systems were
connected by relationships of the form $q=n\times u$, where $q$
is a magnitude, $n$ a number and $u$ a magnitude of the same kind
as $q$, foreboding from the distant past Maxwell's conceptualization
of a physical quantity, although Euclid did not define magnitudes
in terms of units and numbers. 

The modern theory of numbers dramatically extends the theory of numbers
in the \emph{Elements}. Many types of numbers other than positive
integers have been added, and the notion of a number as an element
of an algebraic system has come to the forefront. The modern notion
of number was not developed by a straight-forward extension of the
concept of \emph{arithmos}, however; the initial development of the
new notion of number during the Renaissance was strongly inspired
by the ancient theory of magnitudes.

The beginning of the Renaissance saw renewed interest in the classical
Greek theories of magnitudes and numbers as known from Euclid's \emph{Elements},
but later these two notions gradually fused into that of a real number.
Malet \cite{MAL} remarks:
\begin{quotation}
{\footnotesize{}As far as we know, not only was the neat and consistent
separation between the Euclidean notions of numbers and magnitudes
preserved in Latin medieval translations {[}...{]}, but these notions
were still regularly taught in the major schools of Western Europe
in the second half of the 15th century. By the second half of the
17th century, however, the distinction between the classical notions
of (natural) numbers and continuous geometrical magnitudes was largely
gone, as were the notions themselves.}{\footnotesize \par}
\end{quotation}
The force driving this transformation was the need for a continuum
of numbers as a basis for computation; the discrete \emph{arithmoi}
were not sufficient. As magnitudes of the same kind form a continuum,
the idea emerged that numbers should be regarded as an aspect of magnitudes.
''Number is to magnitude as wetness is to water'' said Stevin in
\emph{L'Arithmétique} \cite{STEV}, published 1585, and defined a
number as ''cela, par lequel s\textquoteright explique la quantité
de chascune chose\textquotedblright{} (that by which one can tell
the quantity of anything). Thus, numbers were seen to form a continuum
by virtue of their intimate association with magnitudes.

Stevin's definition of number is rather vague, and it is difficult
to see how a magnitude can be associated with a definite number, considering
that the numerical measure of a magnitude depends on a choice of a
unit magnitude. The notion of number was, however, refined during
the 17th century. In \emph{La Geometrie} \cite{DESC}, where Descartes
laid the groundwork for analytic geometry, he implicitly identified
numbers with \emph{ratios} of two magnitudes, namely lengths of line
segments, one of which was considered to have unit length,\footnote{In Greek mathematics, the product of a length and a length was an
area, but Descartes argued in \emph{La Geometrie} that it could also
be another length. Descartes did not really multiply lengths, however;
he multiplied the ratios of two lengths $\ell_{1},\ell_{2}$ to a
fixed length $\ell_{0}$ to obtain the ratio of a third length $\ell$
to $\ell_{0}$, using a geometrical construction with similar triangles
such that the number representing the ratio of $\ell$ to $\ell_{0}$
became equal to the product of the number representing the ratio of
$\ell_{1}$ to $\ell_{0}$ and the number representing the ratio of
$\ell_{2}$ to $\ell_{0}$. (See the first figure in \emph{La Geometrie}.)} and in \emph{Universal Arithmetick} \cite{NEWT},\footnote{This was a translation of the Latin original \emph{Arithmetica Universalis},
first printed in Cambridge in 1707 and based on lecture notes by Newton
for the period 1673 to 1683. } Newton, who had studied Descartes thoroughly, defined a number as
follows:
\begin{quote}
{\footnotesize{}By }\emph{\footnotesize{}Number}{\footnotesize{} we
mean, not so much a Multitude of Unities, as the abstracted }\emph{\footnotesize{}Ratio}{\footnotesize{}
of any Quantity, to another Quantity of the same Kind, which we take
for Unity.}{\footnotesize \par}
\end{quote}
In modern terminology, a ratio of quantities of the same kind is a
'dimensionless' quantity. Systems of such quantities contain a canonical
unit quantity $\boldsymbol{1}$, and addition, subtraction, multiplication
and division of dimensionless quantities yield dimensionless quantities.
Hence, a number and the corresponding dimensionless quantity are quite
similar; if there is a difference then Newton's ''abstracted'' ratios
of quantities are numbers, not quantities. Quantities, especially
'dimensionful' quantities, classical magnitudes, were thus needed
only as a scaffolding for the new notion of numbers, and when this
notion had been established its origins fell into oblivion and magnitudes
fell out of fashion. The tradition from Euclid paled away, but the
idea that numbers specify quantities relative to other quantities
remained.

While the Greek theory of magnitudes derived from geometry, the new
theory of quantities served the needs of modern mathematical physics,
which flourished from the second half of the 18th century. In Section
IX, Chapter II of \emph{Théorie analytique de la Chaleur} \cite{FOUR},
Fourier explained how physical quantities related to the numbers in
his equations:
\begin{quotation}
{\footnotesize{}Pour mesurer {[}des quantités qui entrent dans notre
analyse{]} et les exprimer en nombre on les compare a diverses sortes
d'unités, au nombre de cinq, savoir : l'unités de longueur, l'unités
de temps, celle de la temperature, celle du poids, et enfin l'unité
qui sert a mesurer les quantités de chaleur.}{\footnotesize \par}
\end{quotation}
We recognize here the ideas that there are quantities of different
kinds and that the number associated with a quantity depends on the
choice of a unit quantity of the same kind. Fourier also introduced
the powerful notion of \emph{``dimensions''} of quantities. 

Using the modern notion of, for example, a real number, we can generalize
relationships of the form $q=n\times u$, where $n$ is an \emph{arithmos}
and $u$ is a magnitude that measures (divides) $q$, to relationships
of the form $q=\mu\cdot u$, where $u$ is a freely chosen unit quantity
of the same kind as $q$ and $\mu$ is a number specifying the size
of $q$ relative to $u$, the measure of $q$ relative to $u$.\footnote{Thus, originally a measure was a magnitude, but this term is now technically
used to refer to a number, although a measure can still mean ``a
standard or unit of measurement''.} We may write $\mu=f\left(q,u\right)$, noting that with any two of
$\mu$, $q$ and $u$ given, the third is uniquely determined.

Fourier realized that the measure of a quantity may be defined in
terms of measures of other quantities, in turn dependent on the units
for these quantities. For example, the measure of a velocity depends
on a unit of length and a unit of time since a velocity is defined
in terms of a length and a time, and the measure of an area indirectly
depends on a unit of length.

Specifically, let the measure $\mu_{v}$ of a velocity $v$ be given
by $\mu_{v}=F\left(\mu_{\ell},\mu_{t}\right)=F\left(f_{\ell}\left(\ell,u_{\ell}\right),f_{t}\left(t,u_{t}\right)\right)$,
where $F\left(x,y\right)=xy^{-1}$. Generalizing the magnitude identity
$m\times\left(n\times u\right)=mn\times u$, we have $\mu\cdot\left(\nu\cdot u\right)=\mu\nu\cdot u$.
Thus, if $\lambda\neq0$ then $\lambda\mu\cdot\left(\lambda^{-1}\cdot u\right)=\mu\cdot u$,
so $f\left(q,\lambda^{-1}\cdot u\right)=\lambda\mu=\lambda f\left(q,u\right)$,
so it follows from the definition of $F$ that
\[
LT^{-1}\mu_{v}=F\left(f_{\ell}\left(\ell,L^{-1}\cdot u_{\ell}\right),f_{t}\left(t,T^{-1}\cdot u_{t}\right)\right)
\]
for any non-zero numbers $L,T$. Also, let the measure $\mu_{a}$
of the area $a$ of a rectangle be given by $\mu_{a}=G\left(\mu_{\ell},\mu_{w}\right)=G\left(f_{\ell}\left(\ell,u_{\ell}\right),f_{\ell}\left(w,u_{\ell}\right)\right)$,
where $G\left(x,y\right)=xy$. Then, 
\[
L^{2}\mu_{a}=G\left(f_{\ell}\left(\ell,L^{-1}\cdot u_{\ell}\right),f_{\ell}\left(w,L^{-1}\cdot u_{\ell}\right)\right).
\]
Fourier pointed out that quantity terms can be equal or combined by
addition or subtraction only if they agree with respect to each \emph{exposant
de dimension}, having identical patterns of exponents in expressions
such as $LT^{-1}$, $LT^{-2}$ or $L^{2}$, since otherwise the validity
of numerical equations corresponding to quantity equations would depend
on an arbitrary choice of units. He thus introduced the principle
of dimensional homogeneity for equations that contain quantities.

Note that if $q=\mu\cdot u$ then $\lambda\cdot q=\lambda\cdot\left(\mu\cdot u\right)=\lambda\mu\cdot u$,
so $f\left(\lambda\cdot q,u\right)=\lambda\mu=\lambda f\left(q,u\right)$,
where $u$ is a fixed unit. Thus, in a sense turning Fourier's argument
around, we obtain the equations
\begin{gather*}
LT^{-1}\mu_{v}=F\left(f_{\ell}\left(L\cdot\ell,u_{\ell}\right),f_{t}\left(T\cdot t,u_{t}\right)\right),\\
L^{2}\mu_{a}=G\left(f_{t}\left(L\cdot\ell,u_{\ell}\right),f_{\ell}\left(L\cdot w,u_{\ell}\right)\right).
\end{gather*}
Eliminating the fixed units $u_{\ell}$ and $u_{t}$, these equations
can be written as
\begin{gather*}
\varPhi\left(L\cdot\ell,T\cdot t\right)=LT^{-1}\varPhi\left(\ell,t\right),\\
\varGamma\left(L\cdot\ell,L\cdot w\right)=L^{2}\varGamma\left(\ell,w\right),
\end{gather*}
and we can regard $\varPhi$ and $\varGamma$ as quantity-valued functions
because the units for velocity and area are fixed since they depend
on the fixed units for length and time. The bilinearity properties
of $\varPhi$ and $\varGamma$ suggest that we write $\varPhi\left(\ell,t\right)$
as $\ell t^{-1}$ and $\varGamma\left(\ell,w\right)$ as $\ell w$.
Generalizing this, we may introduce the idea that quantities of the
same or different kinds can be multiplied and divided, suggesting
that we can form arbitrary expressions of the form $\prod_{i=1}^{n}q_{i}^{k_{i}},$
where $k_{i}$ are integers. It should be emphasized, however, that
Fourier did not formally define multiplication of quantities, so we
are talking about a dual line of thought that he could have pursued
but apparently did not.

It is clear that the Greek mathematicians' distinction between numbers
and magnitudes is closely related to the modern distinction between
scalars and quantities. In view of Fourier's contribution, it may
be said that the foundation of a modern quantity calculus incorporating
this distinction and treating quantities as mathematical objects as
real as numbers was laid early in the 19th century. Subsequent progress
in this area of mathematics has not been fast and straight-forward,
however. In his survey from 1994, de Boer noted that the modern theory
of quantities had not yet met its Euclid; he concluded that ''a satisfactory
axiomatic foundation for the quantity calculus'' had not yet been
formulated \cite{BOER}. 

Gowers \cite{GOV1} points out that many mathematical constructs are
not defined directly by describing their essential properties, but
indirectly by \emph{construction-definitions}, specifying constructions
that can be shown to have these properties.\footnote{There are mathematical objects for which only construction-definitions
are available, so that the mathematician's task becomes to find the
properties of these constructions. A major example is the natural
numbers, which were created by God, as Kronecker put it, leaving it
to humans to discover their properties.} For example, an ordered pair $\left(x,y\right)$ may be defined by
a construction-definition as a set $\left\{ x,\left\{ y\right\} \right\} $;
it can be shown that this construction has the required properties,
namely that $\left(x,y\right)=\left(x',y'\right)$ if and only if
$x=x'$ and $y=y'$. Many contemporary formalizations of the notion
of a quantity use definitions relying on constructions, often defining
quantities in terms of (something like) scalar-unit pairs, in the
tradition from Maxwell. (See Sections \ref{sec:1-4} and \ref{subsec:19}
for some specifics.) However, this is rather like defining a vector
as a coordinates-basis pair rather than as an element of a vector
space, the modern definition.

Although magnitudes are illustrated by line segments in the \emph{Elements},
the notion of a magnitude is abstract and general. Remarkably, Euclid
dealt with this notion in a very modern way. While he carefully defined
other important objects such as points, lines and numbers in terms
of inherent properties, there is no statement about what a magnitude
''is''. Instead, magnitudes are characterized by how they relate
to other magnitudes through their roles in a system of magnitudes,
to paraphrase Gowers \cite{GOW2}. 

In the same spirit, that of modern algebra, quantities are defined
in this article simply as elements of a \emph{quantity space}. Thus,
the focus is moved from individual quantities and operations on them
to the systems to which the quantities belong, meaning that the notion
of quantity calculus will give way to that of a quantity space. This
article elaborates on the notion of a quantity space introduced in
\cite{JON1,JON2}.

In the conceptual framework of universal algebra, a quantity space
is just a certain \emph{scalable monoid} $\left(X,\mathsf{\ast},(\omega{}_{\lambda})_{\lambda\in R},1_{Q}\right)$,
where $X$ is the underlying set of the algebra, $\left(X,\mathsf{*},1_{Q}\right)$
is a monoid where we write $\ast\left(x,y\right)$ as $xy$, and $\omega_{\lambda}\left(x\right)$
is a scalar product $\lambda\cdot x$ such that $\lambda$ belongs
to a fixed ring $R$, $x\in X$, $\omega_{1}\left(x\right)=x$ for
all $x\in X$, $\omega_{\lambda}\left(\omega_{\kappa}\left(x\right)\right)=\omega_{\lambda\kappa}\left(x\right)$
for all $\lambda,\kappa\in R$, $x\in X$, and $\omega{}_{\lambda}\left(xy\right)=\omega{}_{\lambda}\left(x\right)\,y=x\,\omega{}_{\lambda}\left(y\right)$
for all $\lambda\in R$, $x,y\in X$. 

A scalable monoid $X$ is partitioned into \emph{orbit classes}, which
are equivalence classes with respect to the relation $\sim$ defined
by $x\sim y$ if and only if $\omega_{\alpha}\left(x\right)=\omega_{\beta}\left(y\right)$
for some $\alpha,\beta\in R$. There is no global operation $\left(x,y\right)\mapsto x+y$
defined on $X$, but within each orbit class that contains a unit
element addition of its elements is induced by the addition in $R$,
and multiplication of orbit classes is induced by the multiplication
of elements of $X$.

Quantity spaces are to scalable monoids as vector spaces are to modules.
A quantity space is a \emph{strongly free commutative} scalable monoid
over a \emph{field} $K$. As mentioned, quantities are just elements
of quantity spaces, and dimensions are their orbit classes.

The rest of this article is divided into three parts. Part 1 provides
additional background, Part 2 deals with scalable monoids, and in
Part 3 scalable monoids are specialized to quantity spaces.

\part{From pure numbers to abstract quantities}

Sections 1 and 2 below aim to clarify relations between pure and concrete
numbers, concrete numbers and quantities, and concrete and abstract
quantities. The chain of notions from concrete entities to measures
via concrete and abstract quantities is also considered. Sections
3 and 4 contains a review of previous research on systems of abstract
quantities, leading up to a list of twelve properties of such systems
and three axioms from which these properties can be derived.

\section{\label{sec:1-1}Concrete numbers and quantities}

The numbers defined by Newton were \emph{pure }or \emph{abstract}
numbers, without any designation, but there has always been a practical
need for a notion that addresses questions about how many or how much
\emph{of what}. Reasoning in terms of magnitudes or quantities makes
it possible to answer such questions, but there is an alternative
notion which rose to prominence simultaneously with the demise of
magnitudes \cite{SMIT}, namely \emph{concrete }numbers, that is,
numbers associated with the things being counted or measured. For
example, the Second Part of Euler's \emph{Einleitung zur Rechen-Kunst}
\cite{EUL} deals with such numbers, also known as \emph{benannten
Zahlen,} in some detail.

To understand the relationship between concrete numbers and quantities,
let us consider how one can use them to distinguish, for example,
1 centimeter ($\unit[1]{cm}$) from 1 meter ($\unit[1]{m}$) or 1
gram ($\unit[1]{g}$). For this purpose, we can take advantage of
the fact that quantities are of different kinds. If we interpret $\unit[1]{cm}$,
$\unit[1]{m}$ and $\unit[1]{g}$ as the products $1\cdot cm$, $1\cdot m$
and $1\cdot g$, respectively, where $cm$ and $m$, on the one hand,
and $g$, on the other hand, are quantities of different kinds, then
$1\cdot cm$ and $1\cdot g$ are also quantities of different kinds,
as are $1\cdot m$ and $1\cdot g$, so they are distinct although
the associated pure number 1 is the same in both cases. Note that
we have, for example, $1\cdot cm+2\cdot cm=3\cdot cm$, $1\cdot m+50\cdot cm=1\cdot\left(100\cdot cm\right)+50\cdot cm=100\cdot cm+50\cdot cm=150\cdot cm$
and $\left(1\cdot cm\right)\left(2\cdot g\right)=2\cdot cm\,g$.

However, we can introduce a notion of different\emph{ kinds of numbers}
instead of relying on the fact that quantities are of different kinds.
For example, we may say that $\unit[1]{cm}$, $\unit[1]{m}$ and $1\,\mathsf{g}$
are different kinds of numbers and thus different concrete numbers.
To set this interpretation of $\unit[1]{cm}$, $1\unit{m}$ and $\unit[1]{g}$
apart from the previous one, we may write these expressions as $1_{cm}$,
$1_{m}$ and $1_{g}$, respectively, where the subscript indicates
what kind of number the concrete number is. This interpretation works
smoothly for addition and subtraction of numbers of the same kind;
for example, $1_{cm}+2_{cm}=3_{cm}$. Multiplication and division
of concrete numbers by pure numbers is also straight-forward; for
example, $3\left(2_{m}\right)=6_{m}$ and $\nicefrac{2_{m}}{3}=\left(\nicefrac{2}{3}\right){}_{m}$.
There are also identities such as $m=100\,cm$ and $100_{cm}=\text{1}_{\left(100\,cm\right)}$,
so $\text{1}_{m}+50_{cm}=\text{1}_{\left(100\,cm\right)}+50_{cm}=\text{100}_{cm}+50_{cm}=\text{150}_{cm}$.
Multiplication and division of concrete numbers by concrete numbers
are more complicated operations, however; for example, $\left(3_{cm}\right)\left(2_{g}\right)\!=\!6_{\left(cm\,g\right)}$,
where $cm\,g$ is a new kind of numbers which can be regarded as the
product of the kinds of numbers $cm$ and $g$. Such multiplication
was not considered in Euler's textbook from 1740, and it was apparently
not until the 19th century that products such as $cm\,g$ started
to be used \cite{SMIT}.

If a kind of numbers can be multiplied by a pure number or another
kind of numbers, with the products being kinds of numbers, one gets
the the impression, however, that kinds of numbers are just quantities
in disguise. In fact, to be able to multiply and divide concrete numbers
of any kinds one needs to develop a calculus of kinds of numbers essentially
equivalent to a calculus of quantities, so nothing is gained in terms
of simplicity, and it is conceptually cleaner to separate quantities
from numbers and distinguish different kinds of quantities than to
treat quantities as numbers that can be of different kinds.

\section{\label{sec:1-2}On concrete and abstract quantities}

Recall that while Stevin defined a number as something which is a
measure of a magnitude, Maxwell described a quantity as something
which is measured by a number relative to a unit. Leaving this historical
irony aside, Maxwell's characterization of a quantity as something
which can be specified by a number relative to another quantity of
the same kind is open to two somewhat different interpretations. On
the one hand, the relation $q=\left\{ q\right\} \left[u\right]$,
where $\left\{ q\right\} $ is a number and $q$ and $\left[u\right]$
are quantities of the same kind, can be be seen as analogous to $v=\lambda u$,
where $\lambda$ is a scalar and $v$ and $u$ are vectors. That is,
we can see a quantity as a mathematical object which admits multiplication
by a number. Equipping a system of such mathematical objects with
other operations as needed and assuming that certain identities hold,
we obtain a mathematical system of quantities. We may call quantities
in this sense \emph{mathematical }quantities, or alternatively \emph{abstract
}quantities.

On the other hand, Maxwell \cite{MAXW} talks about the quantity $\left[u\right]$
as a ``standard of reference'' and a ``standard quantity {[}which{]}
is technically called the Unit'', and he discusses concrete standards
of length, time and mass such as the standard \emph{mètre} in Paris.
This points to another possible interpretation of Maxwell's notion
of quantity: a \emph{concrete }quantity\emph{,} or alternatively a\emph{
physical }quantity\emph{,} is an \emph{attribute }of a\emph{ concrete
or physical entity,} such as a concrete object, process or phenomenon,
which can be specified by a number relative to (a corresponding attribute
of) a concrete reference entity. For example, the length $\ell\left(r\right)$
of a rigid rod $r$ can be expressed as a multiple $\mu$ of the length
$\ell\left(\epsilon\right)$ of a measuring rod $\epsilon$; we may
write $\ell\left(r\right)=\mu\cdot\ell\left(\epsilon\right)$. More
generally, if we let $a\left(\epsilon\right)$ denote the attribute
of a concrete standard $\epsilon$ with respect to some aspect $a$
then the attribute $a\left(x\right)$ of a concrete entity $x$ with
respect to $a$ can be expressed as a multiple $\mu$ of $a\left(\epsilon\right)$,
that is, $a\left(x\right)=\mu\cdot a\left(\epsilon\right)$.

It is a wide-spread practice to use the notation of concrete numbers
to refer to concrete quantities. Specifically, we may denote $a\left(\epsilon\right)$
by $\unit[1]{a\left(\epsilon\right)}$ or $\unit[1]{u}$, and $\mu\cdot a\left(\epsilon\right)$
by $\unit[\mu]{a\left(\epsilon\right)}$ or $\unit[\mu]{u}$; with
this notation, we have $\mu\cdot\left(\unit[1]{u}\right)=\unit[\mu]{u}$.
By contrast we shall use italics for mathematical quantities; thus,
the mathematical quantity corresponding to $\unit[\mu]{u}$ can be
denoted by $u$. We conclude that a mathematical quantity \emph{can}
be specified by a number $\mu$ relative to a unit $u$ by an expression
of the form $\mu\cdot u$, whereas a concrete quantity $\unit[\mu]{u}$
is an attribute of a concrete entity that \emph{must} be specified
by a number $\mu$ relative to a concrete standard or unit $\unit[1]{u}$.\footnote{That there are two roads to quantities has led to some confusion and
debate among physicists and metrologists, and to attempts at clarification.
In a seminal contribution 1950, König \cite{KON} pointed out the
ambiguity of the notion of quantity and called the two camps each
favoring one of these interpretations \emph{Synthetiker} and \emph{Realists}.
Distinctions similar to that proposed here between different kinds
of quantities were subsequently introduced by several authors. Some
pairs of notions corresponding to abstract and concrete quantities,
respectively, are ''symbolic'' and ''physical'' quantities \cite{SILS},
''(abstract) physical'' and ''concrete'' quantities \cite{BOER},
''abstract'' and ''measurable''\emph{ }quantities \cite{EM},
and ''magnitudes'' and ''quantities'' \cite{COURT}.} 

Given a system of physical quantities, we require a corresponding
system of mathematical quantities to be an abstract model of the former,
meaning that there has to exist some kind of structural similarity
between the two systems. In particular, as mathematical quantities
of the same kind can be added, it should be possible to add certain
concrete quantities as well. As an example, two rigid rods can be
concatenated, ``glued'' together to form one longer, straight rod,
and we may define the sum of the lengths of the two original rods
to be the length of the new rod. We may similarly ``concatenate''
two metal weights (by putting them in the same pan of a balance scale)
or two processes (in a temporal sense), and define the sum of masses
and times in terms of such concatenation. Specifically, 
\[
a\left(x\right)\dotplus a\left(y\right)=a\left(x\,\&\,y\right),
\]
where $x\,\&\,y$ is the concrete entity obtained by concatenating
$x$ and $y$, $a\left(x\,\&\,y\right)$ is the length, mass or duration
of the concatenated rods, weights or processes, and $a\left(x\right)\dotplus a\left(y\right)$
is the sum of the two concrete quantities characterizing the original
concrete entities. It is clear that addition of concrete quantities
inherits some properties from concatenation of the underlying concrete
entities. In particular, we have $\left(x\,\&\,y\right)\,\&\,z=x\,\&\,\left(y\,\&\,z\right)$
and $x\,\&\,y=y\,\&\,x$, so
\[
\left(a\left(x\right)\dotplus a\left(y\right)\right)\dotplus a\left(z\right)=a\left(x\right)\dotplus\left(a\left(y\right)\dotplus a\left(z\right)\right),\quad a\left(x\right)\dotplus a\left(y\right)=a\left(y\right)\dotplus a\left(x\right),
\]
so addition of mathematical quantities should be defined to be associative
and commutative as well.

It is less straight-forward to construct concrete quantities that
can be interpreted as products or inverses of given concrete quantities,
however, and this is something that critics of modern quantity calculus
have focused on, since it assumes that mathematical quantities can
be multiplied and have inverses. Hence, the structural similarity
of abstract and concrete systems of quantities is called into question. 

In 1922, Bridgman wrote in his well-known \emph{Dimensional Analysis}
\cite{BRID} that ``{[}i{]}t is meaningless to talk of dividing a
length by a time: what we actually do is to operate on numbers that
are measures of these quantities''. Much later, Emerson \cite{EM}
argued that ``{[}i{]}t is impossible to conceive of the distance
between the ends of a meter standard, a unit, being divided by the
half period of a seconds pendulum, another unit''. Velocity may be
a badly chosen example, however. While it is not possible to ``divide''
a measuring rod by a pendulum, so that the technique used above to
define sums of certain quantities does not work in this case, we may,
given a length unit $\unit[1]{u_{\ell}}$ and a time unit $\unit[1]{u_{t}}$,
define a velocity unit $\unit[1]{\left(u_{\ell}\div u_{t}\right)}$
as the mean velocity of a point that travels the distance $\unit[1]{u_{\ell}}$
during the time $\unit[1]{u_{t}}$. We can also define an area unit
$\unit[1]{\left(u_{\ell}\divideontimes u_{\ell}\right)}$ as the area
of a square with sides of length $\unit[1]{u_{\ell}}$. Thus, $\unit[1]{\left(u_{\ell}\div u_{t}\right)}$
is an attribute of a movement, and $\unit[1]{\left(u_{\ell}\divideontimes u_{\ell}\right)}$
is an attribute of a square. While the operations $\divideontimes$
and $\div$ in a system of concrete quantities are not identical to
multiplication and division, respectively, in an abstract system of
quantities, there is a correspondence $\phi$ between the two systems
such that $\phi\left(\unit[1]{\left(u_{\ell}\divideontimes u_{\ell}\right)}\right)=\phi\left(\unit[1]{u_{\ell}}\right)\phi\left(\unit[1]{u_{\ell}}\right)=u_{\ell}u_{\ell}$
and $\phi\left(\unit[1]{\left(u_{\ell}\div u_{t}\right)}\right)=\phi\left(\unit[1]{u_{\ell}}\right)\phi\left(\unit[1]{u_{t}}\right)^{-1}=u_{\ell}u_{t}^{-1}$,
so a structural similarity exists.

The critics may be on to something nevertheless. It is not clear,
for example, what kind of physical entity a physical time unit multiplied
by itself, or a concrete mass unit multiplied by itself, would be
an attribute of. More generally, it is not clear that any two concrete
quantities can be combined in a way that corresponds to multiplication
of abstract quantities, and the notion of the inverse of a concrete
quantity is also not clear. Thus, the correspondence $\phi$ may be
a partial function only. On the other hand, a partial function may
establish a sufficient structural similarity between systems of physical
and mathematical quantities.

We will have to leave this difficulty unresolved and unexplored here,
though; a full discussion would require another article (and, I believe,
another author). Instead, we will try to gain a deeper understanding
of the situation by considering an analogy with a more familiar case.

Recall that abstract vectors may be represented by coordinates relative
to some basis. On the other hand, abstract vectors may represent different
kinds of less abstract objects, such as geometrical vectors, characterized
by their direction and extension. (As we know, ``vectors'' originally
meant geometrical vectors.) Furthermore, geometrical vectors may themselves
be seen as abstract representations of more concrete geometrical phenomena
such as translations or oriented line segments. These facts can be
visualized as shown below:
\noindent \begin{center}
\emph{\small{}Concrete geometrical entity $\rightarrow$ Geometrical
vector $\rightarrow$ Abstract vector $\rightarrow$ Coordinates}
\par\end{center}{\small \par}

For example, the set $T$ of all translations in a Euclidean space
is an abelian group under composition of transformations. Thus, the
corresponding set $G$ of geometrical vectors is an abelian group
under the operation inherited from the set of translations, so $G$
is isomorphic to some abstract vector space $V$ as an abelian group.
A scalar product $\left(\lambda,x\right)\mapsto\lambda\cdot x$ can
also be defined on $T$, making $G$ with this scalar product inherited
from $T$ a vector space isomorphic to $V$.

On the other hand, if we let the concrete geometrical entities be
oriented line segments then the ``sum'' of $\overrightarrow{AB}$
and $\overrightarrow{CD}$ is defined and equal to $\overrightarrow{AD}$
only if $B=C$, so while each oriented line segment can be interpreted
as a geometrical vector, the set of all oriented line segments does
not form an abelian group of geometrical vectors. Every oriented line
segment $\overrightarrow{AB}$ can be associated with the unique translation
that moves $A$ to $B$, however, so there is still a certain structural
similarity between the system of oriented line segments and the abelian
group $G$ of geometrical vectors corresponding to the set of translations
$T$. Relations that connect concrete geometrical entities to geometrical
vectors need not be strict isomorphisms, though, and this applies
to relations between physical entities and physical quantities, too.

Returning to the quantities, we have seen that we can distinguish
four levels of notions, corresponding to those in our vector space
analogy, as shown below:
\noindent \begin{center}
\emph{\small{}Concrete entity $\rightarrow$ Concrete quantity $\rightarrow$
Abstract quantity $\rightarrow$ Measure}
\par\end{center}{\small \par}

Although the arrows between the three left-most notions possibly hide
not fully resolved complications, this simple scheme helps to clarify
some central conceptual issues concerning quantities. It would seem
that leaving out either concrete quantities or abstract quantities,
using instead one of the conceptual schemes 
\noindent \begin{center}
\emph{\small{}Concrete entity $\rightarrow$ Concrete quantity $\rightarrow$
Measure,}\\
\emph{\small{}Concrete entity $\rightarrow$ Abstract quantity $\rightarrow$
Measure, }
\par\end{center}{\small \par}

\noindent or, even worse, mixing the two schemes, leads to a conceptually
less satisfying theory, if not conceptual confusion \cite{SILS}.
\begin{example}
Many-to-one relations between concrete and abstract quantities.
\end{example}
The definition of planar angle quantities has been subject to much
debate in the metrological community \cite{BOER2}: should a planar
angle be treated as a derived quantity with unit quantity $\unitfrac[1]{m}{m}$,
or should it be treated as a base quantity with unit quantity $\unit[1]{rad}$
(radian), where the radian is the angle subtended at the center of
a circle by an arc that is equal in length to the radius of the circle?
That depends on whether we treat a planar angle as a physical or a
mathematical quantity. As a physical quantity, a planar angle $\unit[\theta]{rad}$
is one expressed by a number and a physical unit quantity, a planar
angle given by the geometrical construction referred to. However,
$\unitfrac[\theta]{m}{m}$ is not an attribute of any particular physical
entity; in this context, $\unitfrac[\theta]{m}{m}$ is the physical
quantity $\unit[\theta]{rad}$ treated as a mathematical quantity,
and therefore it should be written as $\theta\cdot\nicefrac{m}{m}$
or $\theta\cdot\boldsymbol{1}$, meaning that a planar angle as a
mathematical quantity is a so-called dimensionless quantity. 

Similar considerations apply to many other quantities. For example,
there has been a parallel debate about whether a solid angle should
be treated as a derived quantity with unit quantity $\unitfrac[1]{m^{2}}{m^{2}}$,
or as a base quantity with unit quantity $\unit[1]{sr}$ (steradian).
Again, as a physical quantity, a solid angle $\unit[\theta]{sr}$
is one expressed by a number and a certain physical solid angle unit,
while as a mathematical quantity a solid angle is again a dimensionless
quantity $\theta\cdot\nicefrac{m^{2}}{m^{2}}=\theta\cdot\boldsymbol{1}$. 

Thus, regarded as physical quantities, planar and solid angles are
quantities of different kinds; but their corresponding mathematical
quantities are of the same kind. This shows that a correspondence
between concrete and abstract quantities may be a many-to-one relation.
\begin{example}
Quantities and quantity vectors.
\end{example}
Many ``quantities'' in physics are vector quantities or, more correctly,
\emph{quantity vectors}; they are geometrical vectors that combine
a direction and an extent given by a quantity rather than a number. 

For example, a displacement $D$ combines a direction with a length
$\unit[\ell]{m}$ that gives the extension of the displacement, while
a force $F$ combines a direction with a quantity $\unit[f]{N}$ that
gives the extent (strength) of the force. The \emph{work} generated
by $F$ and $D$ is defined as the \emph{quantity $FD=\unit[f\ell\cos\theta]{N\,m}$,
}where $\theta$ is the oriented angle between $F$ and $D$.

On the other hand, we can define a \emph{quantity vector} $F\times D$
with the same direction as that of the cross product of $F$ and $D$
and extent $f\unit[\ell\sin\theta]{N\,m}$; this is the \emph{torque}
created by $F$ and $D$.\footnote{The torque can also be defined as the geometrical bivector $F\wedge D$
whose extent is $f\unit[\ell\sin\theta]{N\,m}$.} 

Note that $\unit[1]{N\,m}$ is a unit for the physical quantity $FD=\unit[f\ell\cos\theta]{N\,m}$
as well as the physical quantity $f\unit[\ell\sin\theta]{N\,m}$ associated
with the quantity vector $F\times D$, so these physical quantities
are of the same kind. Thus, although work and torque are often said
to illustrate a situation where ''quantities having the same quantity
dimension are not necessarily of the same kind'' \cite{JCGM}, they
do not exemplify physical quantities of different kinds corresponding
to mathematical quantities of the same kind, at least if we accept
that the torque is a quantity vector rather than a vector quantity.
\begin{example}
Fractional powers of quantities.

In quantity equations we sometimes encounter fractional powers of
quantities. For example, the Gaussian unit of electrical charge $\unit[1]{statC}$
is usually defined by 
\begin{equation}
\unit[1]{statC}=1\frac{\unit{cm}^{3/2}\,\unit{g}^{1/2}}{\unit{s}}.\label{eq:ex1}
\end{equation}
Do we thus have to admit fractional powers of concrete and abstract
quantities into the theory of quantities? Several authors do so, but
this complicates the interpretation of concrete quantities and the
definition of abstract quantities. Fortunately, integral powers of
quantities suffice. 

Let $c$, $g$, $\ell$ and $t$ be the abstract quantities corresponding
to the concrete quantities $\unit[1]{statC}$, $\unit[1]{g}$, $\unit[1]{cm}$
and $\unit[1]{s}$, respectively. Suppose that we have
\begin{equation}
c^{2}=\frac{\ell^{3}g}{t^{2}},\label{eq:ex2}
\end{equation}
where fractional powers of quantities do not occur. If the measures
of these quantities are positive, we can write this as
\begin{equation}
\kappa=\frac{\lambda^{3/2}\,\gamma{}^{1/2}}{\tau},\label{eq:ex3}
\end{equation}
where $\kappa$, $\lambda$, $\gamma$ and $\tau$ are the measures
of $c$, $\ell$, $g$ and $t$, respectively, relative to some coherent
system of units, so we have effectively recovered (\ref{eq:ex1}).
The crucial notion here is that the relation between the concrete
quantities $\unit[1]{statC},$ $\unit[1]{cm},$ $\unit[1]{g}$ and
$\unit[1]{s}$ is described by the identity (\ref{eq:ex2}) relating
the corresponding abstract quantities; (\ref{eq:ex1}) does not describe
this relation directly. This example suggests that there is no need
for fractional powers of quantities in a theory where a distinction
is made between concrete and abstract quantities.
\end{example}

\section{\label{sec:1-3}Algebraic properties of systems of abstract quantities}

Some years after Maxwell had called attention to the notion of quantities,
Lodge published a note on ``the multiplication and division of concrete
quantities'' \cite{LOD}. In retrospect, Lodge's note hints at some
ideas that reappear in modern quantity calculus. Using motivating
examples such as 
\[
\frac{\unit[4\frac{1}{2}]{miles}}{\unit[1]{hour}}\times\unit[40]{minutes}=\frac{\unit[4\frac{1}{2}]{miles}}{3}\times2=\unit[3]{miles}
\]
and 
\[
\unit[4]{feet}\times\unit[2]{yards}=8\times\unit[1]{foot}\times\unit[1]{yard}=24\times\unit[1]{foot}\times\unit[1]{foot}=\unit[24]{square\:feet},
\]
Lodge assumes that concrete quantities can be multiplied and divided
by numbers and by other concrete quantities, and sets out to investigate
the formal properties of these operations. Although Lodge deals with
concrete quantities of the form $\unit[\alpha]{u}$, he takes a step
towards a more abstract notion of quantities in this investigation
by denoting such concrete quantities by single symbols such as $u$.
A product of a concrete quantity by a number can then be written as
$\lambda\cdot u$, and products or ratios of concrete quantities can
be written as $uv$ or $u/v$.

In the second example above, it is assumed that multiplication of
quantities by numbers is associative in the sense that $\alpha\cdot\left(\beta\cdot\left(\unit[1]{u}\right)\right)=\alpha\beta\cdot\left(\unit[1]{u}\right)$
for any numbers $\alpha,\beta$ and any quantity $\unit[1]{u}$. Hence,
$\alpha\cdot\unit[\beta]{u}=\unit[\alpha\beta]{u}$ since $\gamma\cdot\left(\unit[1]{u}\right)=\unit[\gamma]{u}$,
so 
\begin{gather}
1\cdot u=1\cdot\unit[\alpha]{u}=\unit[\alpha]{u}=u,\label{eq:1a}\\
\alpha\cdot\left(\beta\cdot u\right)=\alpha\cdot\left(\beta\cdot\unit[\gamma]{u}\right)=\alpha\cdot\left(\unit[\beta\gamma]{u}\right)=\unit[\alpha\beta\gamma]{u}=\alpha\beta\cdot\left(\unit[\gamma]{u}\right)=\alpha\beta\cdot u,\label{eq:1b}
\end{gather}
although these identities are not stated explicitly. 

Regarding multiplication and division of two quantities $\alpha\cdot a$
and $\beta\cdot b$, Lodge is more explicit and argues that these
operations have to obey the laws
\begin{gather}
\left(\alpha\cdot a\right)\left(\beta\cdot b\right)=\alpha\beta\cdot ab,\label{eq:1}\\
\left(\alpha\cdot a\right)/\left(\beta\cdot b\right)=\left(\alpha/\beta\right)\cdot\left(a/b\right),\label{eq:2}
\end{gather}
presuming that $a/b$ and $\alpha/\beta$ exist. It is worth noting
that (\ref{eq:1}) is equivalent to
\begin{equation}
\lambda\cdot ab=\left(\lambda\cdot a\right)b=a\left(\lambda\cdot b\right);\label{eq:3}
\end{equation}
if (\ref{eq:1a}) and (\ref{eq:1}) then trivially (\ref{eq:3}),
and if (\ref{eq:1b}) and (\ref{eq:3}) then (\ref{eq:1}), as shown
in the proof of Lemma \ref{thm:lem1}.

Lodge suggested that ''{[}quantities{]} might be treated exactly
as numbers''; this may be taken to mean that concrete quantities
should be treated as mathematical quantities \textendash{} for example,
one should write $\frac{\unit[3]{feet}}{\unit[1]{second}}$ rather
than $\unit[3]{feet\:per\:second}$. Note, though, that this is a
heuristic principle rather than a rigorous theory: physical quantities
are not mathematical quantities, and mathematical quantities are not
numbers. 

After Lodge, no significant progress was made in the development of
a rigorous theory of mathematical quantities for quite some time.
In 1945, Landolt \cite{LAND}, supposedly encouraged by Wallot's promotion
of quantity calculus in \emph{Handbuch der Physik} \cite{WAL} and
inspired by the abstract ``modern algebra'' expounded in van der
Waerden's classical textbook \cite{WAERD}, called attention to group
operations on systems of quantities. Specifically, he pointed out
that the invertible quantities form a group under \emph{``qualitative
Verknüpfung''}, that is, multiplication of quantities as discussed
by Lodge, and that quantities of the same kind form a group under
\emph{``intensive Verknüpfung''}, that is, addition of quantities.

After Landolt, but apparently independent of him, Fleischmann published
an important article focusing on multiplication of \emph{kinds of
quantities} (\emph{Grössen\-arten}). In modern terminology, Fleischmann's
idea is that a congruence relation $\sim$ can be defined on a given
set of quantities, so that the kinds of quantities correspond to the
equivalence classes $\left[q\right]$, and multiplication of kinds
of quantities can be defined by $\left[p\right]\left[q\right]=\left[pq\right]$.

Fleischmann states seven\emph{ Grundaussagen} about kinds of quantities.
First, there are four assertions to the effect that they constitute
a group. Then it is asserted that this group is abelian, torsion-free
and finitely generated. These assertions imply that the kinds of quantities
form a finitely generated free abelian group, so there exists a finite
basis for this group. This means that if $\left\{ B_{1},\ldots,B_{n}\right\} $
is such a basis then any kind of quantity $Q$ has a unique representation
of the form $Q=\prod_{i=1}^{n}B_{i}^{k_{i}}$, where $k_{i}$ are
integers.

In an Appendix to his survey from 1995, de Boer \cite{BOER} presented
a list of properties of systems of abstract quantities, based on previous
work on quantity calculus and meant to be exhaustive. The list of
properties below is, in turn, based on de Boers list, but there are
also elements from other sources, such as the works by Lodge, Landolt
and Fleischmann surveyed in this section. 
\begin{enumerate}
\item[(P1)] A quantity $q$ can be multiplied by a number $\lambda$, and we
have the identities $1\cdot q=q$ and $\alpha\cdot\left(\beta\cdot q\right)=\left(\alpha\beta\right)\cdot q$.
\item[(P2)] A set of quantities that forms a system of abstract quantities can
be partitioned into subsets of quantities of the same kind.
\item[(P3)] $\lambda\cdot q$ is a quantity of the same kind as $q$.
\item[(P4)] Quantities of the same kind can be added and form an abelian group
under addition.
\item[(P5)] If $p$ and $q$ are of the same kind then $\lambda\cdot\left(p+q\right)=\lambda\cdot p+\lambda\cdot q$
and $\left(\alpha+\beta\right)\cdot q=\alpha\cdot q+\beta\cdot q$.
\item[(P6)] A quantity can be multiplied by a quantity, and the non-zero quantities
form an abelian group under multiplication. 
\item[(P7)] Multiplication of quantities by numbers and by quantities are related
by the identities $\lambda\cdot\left(pq\right)=\left(\lambda\cdot p\right)q$
and $\lambda\cdot\left(pq\right)=p\left(\lambda\cdot q\right)$.
\item[(P8)] Kinds of quantities can be multiplied, forming a finitely generated
free abelian group under multiplication. 
\item[(P9)] If $q$ is a quantity of kind $K$ and $q'$ a quantity of kind $K'$
then $qq'$ is a quantity of kind $KK'$.
\item[(P10)] If $q$ is a quantity and $r,s$ are quantities of the same kind
then $q\left(r+s\right)=qr+qs$ (and hence $\left(r+s\right)q=rq+sq$).
\item[(P11)] For every kind of quantities there is a quantity $u$ such that for
every quantity $q$ of this kind we have $q=\mu\cdot u$, where $\mu$
is a uniquely determined number. Such a quantity $u$ is called a
\emph{unit}.
\item[(P12)] It is possible to select exactly one unit from each kind of quantities
in such a way that that if $u$ is the unit of kind $K$ and $u'$
is the unit of kind $K'$ then $uu'$ is the unit of kind $KK'$.
A set of units satisfying this condition is said to be \emph{coherent}.
A coherent set of units is a finitely generated free abelian group,
isomorphic to the corresponding group of kinds of quantities.
\end{enumerate}
These properties have many implications. For example, as a coherent
set $U$ of units is a finitely generated free abelian group under
multiplication, there is a basis $\left\{ u_{1},\ldots,u_{n}\right\} $
for $U$ such that each unit $u$ in $U$ has a unique representation
of the form $u=\prod_{i=1}^{n}u_{i}^{k_{i}}$, where $k_{i}$ are
integers. The units $u_{1},\ldots,u_{n}$ are called ``base units'';
other units are called ``derived units''. In addition, there exists
a quantity $\boldsymbol{1}$ such that $\boldsymbol{1}q\!=\!q\boldsymbol{1}=\!q$
for any quantity $q$, there exists a kind of quantities $I$ such
that $IK\!=\!KI=\!K$ for any kind of quantity $K$, $\boldsymbol{1}$
is of kind $I$, and there is a canonical bijection $\lambda\mapsto\lambda\cdot\boldsymbol{1}$
between the set of numbers and the quantities of kind $I$. 

Note that although (P10) is not mentioned by Lodge, Landolt, Fleischmann
or de Boer, it is postulated by Quade \cite{QUAD} and Raposo \cite{RAP}
and derived from other properties by Whitney \cite{WHIT}. Conversely,
though, (P1) \textendash{} (P12) include or imply all items in de
Boer's list.

\section{\label{sec:1-4}Systems of abstract quantities from first principles}

Given an exhaustive list of properties of systems of mathematical
quantities, it remains to design some kind of mathematical foundation
from which these properties can be derived, a process analogous to
Newton's (ostentative) derivation of Kepler's laws from first principles
or Euclid's axiomatization of Greek geometry. In reality, description
and derivation were intertwined in the development of quantity calculus,
but for purposes of exposition I have somewhat artificially separated
these processes, focusing on the former in the preceding section and
the latter here. We can let the (P1) \textendash{} (P12) be the properties
of quantity systems that we want to derive.

Abstract quantities will be referred to simply as quantities below
since only abstract quantities will be considered henceforth.

Carlson \cite{CARL} defines a quantity as a pair $\left(r,\xi\right)$,
where $r$ is a real number and $\xi$ a \emph{pre-unit}, an element
of a multiplicatively written vector space over the rational numbers.
Multiplication of quantities by real numbers (i) is defined by setting
$a\cdot\left(r,\xi\right)=\left(ar,\xi\right)$, multiplication of
quantities (ii) is defined by setting $\left(r,\xi\right)\left(s,\eta\right)=\left(rs,\xi\eta\right)$
and fractional powers of quantities (iii) are defined by setting $\left(r,\xi\right)^{m/n}=\left(\left(r^{m}\right)^{1/n},\xi^{m/n}\right)$
if a (unique) positive real \emph{n}th root $\left(r^{m}\right)^{1/n}$
of $r^{m}$ exists.

Recall that Fleischmann's \emph{Grundaussagen} imply that there exists
a finite basis $\left\{ B_{1},\ldots,B_{n}\right\} $ for the group
of kinds of quantities, meaning that every kind of quantity $Q$ has
a unique representation of the form $Q=\prod_{i=1}^{n}B_{i}^{k_{i}}$,
where $k_{i}$ are integers. It similarly follows from axioms given
by Carlson that for every system of quantities there exists a\emph{
fundamental system of units} $\left\{ u_{1},\ldots,u_{n}\right\} $,
where $u_{i}^{t}$ is defined for each $u_{i}$ and every rational
number $t$, such that every quantity $q$ has a unique expansion
$q=\mu\cdot\prod_{i=1}^{n}u_{i}^{t_{i}}$, where $t_{i}\in\mathbb{Q}$
is said to be a \emph{dimension} of $q$. A fundamental system of
units is analogous to a basis for a vector space, and is not unique.
Carlson points out, however, that if the dimensions of two quantities
agree relative to one fundamental system of units, then they agree
relative to all fundamental systems of units, and he defines \emph{quantities
of the same kind} as quantities whose dimensions agree relative to
some fundamental system of units, although he does not consider addition
and subtraction of quantities of the same kind.

The definitions proposed by Drobot \cite{DROB} and Whitney \cite{WHIT}
are both based on the ideas that the set of quantities \textendash{}
rather than a set of pre-units \textendash{} is a multiplicatively
written vector space (over $\mathbb{R}$ for Drobot, $\mathbb{Q}$
or $\mathbb{R}$ for Whitney), and that this vector space contains
a set $\mathcal{R}$ of numbers, so that multiplication of type (i)
is reduced to multiplication of type (ii). The exponentiation operation
(iii) is present in both authors' definitions; Drobot lets $\mathcal{R}$
be the set of positive real numbers so that $r^{t}$ is always defined,
but Whitney lets $\mathcal{R}$ be all rational or real numbers and
$r^{t}$ be undefined for non-positive numbers $r$, so in his theory
the set of quantities is not quite a multiplicatively written vector
space but something more complicated.

Both Drobot and Whitney define, in slightly different ways, sets of
quantities of the same kind, called ``dimensions'' by Drobot and
``birays'' by Whitney,\footnote{We note that the term ''dimension'' is used to refer to a pattern
of exponents of numbers by Fourier \cite{FOUR}, an exponent of a
unit by Carlson \cite{CARL}, a set of quantities by Drobot \cite{DROB},
and an object associated with a set of quantities by Raposo \cite{RAP}.
In this article, a dimension will be defined as a set of quantities.} and both define addition of quantities $x=\alpha u$ and $y=\beta u$
of the same kind by the identity $x+y=\left(\alpha+\beta\right)u$.
Letting $\left[q\right]$ denote the dimension/biray containing $q$,
both authors define multiplication of dimensions/birays by the identity
$\left[x\right]\left[y\right]=\left[xy\right]$, and exponentiation
by $\left[x\right]^{t}=\left[x^{t}\right]$. Whitney also proves that
$q\left(r+s\right)=qr+qs$ for for any quantities $q,r,s$ such that
$\left[r\right]=\left[s\right]$.

Drobot's article \cite{DROB} appeared in 1953, Whitney's \cite{WHIT}
in 1968, and Whitney\linebreak{}
acknowledged that ``{[}t{]}he essential features will be found in
Drobot'', but Drobot failed to justify some important definitions
and Whitney's theory is more well-developed. Both theories have peculiar
features related to the overloading of numbers and quantities, however.
In Whitney's theory, the number 0 belongs to all birays; note that
$0x=0$ for any quantity $x$, so this means that there exists a unique
zero quantity which can be regarded as a quantity of all kinds. In
Drobot's theory, there are only ''positive'' quantities, so a bigger
quantity cannot be subtracted from a smaller one, and there is no
zero quantity.

Quade \cite{QUAD} defines systems of quantities by means of a rather
cumbersome construction with \emph{one-dimensional vector spaces}
as building blocks. As a first step, he defines a quantity system
as the union of all vector spaces in a countably infinite set of pairwise
disjoint vector spaces over a fixed field of real or complex numbers;
the elements of these vector spaces, called \emph{V-elements}, are
in effect quantities. Quantities are of the same kind if and only
if they belong to the same vector space; naturally, only quantities
in the same vector space can be added. Quade also assumes that the
set of quantities can be multiplied and form an abelian semigroup,
that for any two one-dimensional vector spaces $U$ and $V$, the
set of products of quantities $uv$, where $u\in U$ and $v\in V$,
is itself a one-dimensional vector space, denoted $UV$, so that the
set of vector spaces is an abelian semigroup under this multiplication.
Furthermore, we have $\lambda\left(uv\right)=\left(\lambda u\right)v$,
where $u,v$ are vectors and $\lambda$ a scalar, and if $u$ and
$v$ are non-zero quantities then $uv$ is a non-zero quantity.

This scheme does not accommodate inverses of quantities and general
division of quantities, however, so Quade embeds the multiplicative
semigroup of quantities into a group of $V$-element fractions, redefining
$V$-elements as such fractions. The set of all generalized $V$-elements
thus obtained can be partitioned into vector spaces, corresponding
to kinds of generalized quantities. The set of these vector spaces
is again an abelian group, now a finitely generated free abelian group.

It is worth noting that Quade introduces a relation $R_{1}$ such
that $xR_{1}y$ if and only if $\alpha x=\beta y$ for some numbers
$\alpha,\beta$, not both zero. $R_{1}$ is similar to the relation
$\sim$ defined in Section \ref{subsec:13} below, but the definitions
are not identical and the relations are used in different ways. 

Recently, Raposo \cite{RAP} has proposed a definition of a system
or ''space'' of quantities which is similar to Quade's but more
concise and elegant. By this definition, a space of quantities $Q$
is an algebraic fiber bundle, with fibers of quantities attached to
dimensions (kinds of quantities) in a base space assumed to be a finitely
generated free abelian group. Each fiber is again a one-dimensional
vector space, with scalar product $\lambda\cdot q$. Multiplication
of quantities and multiplication of dimensions are defined independently,
but are assumed to be compatible in the same sense as for Quade. The
quantities constitute an abelian monoid, and it is assumed that $q(r+s)=qr+qs$
for any quantities $q$ and $r,s$ of the same kind, and that $\lambda(qr)=(\lambda q)r$
for any scalar $\lambda$ and quantity $q$. Although this is technically
not part of the definition of a space of quantities, Raposo also assumes
that if $q$ and $r$ are non-zero quantities then $qr$ is a non-zero
quantity. 

Looking back, we can discern three main approaches to the definition
of quantity systems. Carlson uses a two-component approach: a quantity
is a pair $\left(r,\xi\right)$, where $r$ is a real number and $\xi$
a pre-unit, and operations on quantities are defined in terms of operations
on their components. Quade and Raposo use a fiber-bundle approach.
In Quade's conceptualization, described in fiber-bundle terms, each
element of the base space is identified with the corresponding fiber
and is a one-dimensional vector space whose elements are quantities;
in Raposo's version, the elements of the base space are dimensions
with attached fibers which are again one-dimensional vector spaces
whose elements are quantities. In Drobot's and Whitney's approach,
elements of a subspace $\mathcal{R}$ of a multiplicatively written
vector space $Q$ are identified with numbers, and elements of $Q\setminus\mathcal{R}$
are interpreted as quantities in a narrow sense.

In the three main approaches sketched above, properties of systems
of quantities are derived from somewhat contrived constructions, supplemented
by postulates. Most of the properties of systems of quantities listed
in Section \ref{sec:1-3} can be derived from these definitions, but
the tedious details will not be given here \textendash{} let us look
instead at the definition of quantity spaces proposed in this article.
It turns out, surprisingly, that the identities listed below, corresponding
to items (P1), (P7) and a modified version of (P8) that takes clues
about the definition of a basis from both Fleischmann and Carlson,
suffice to derive (P1) \textendash{} (P12). In other words, a quantity
space may be informally defined as a set of elements, called quantities,
satisfying (A1), (A2) and (A3).
\begin{enumerate}
\item[(A1)]  A quantity can be multiplied by a number, and we have the identities
$1\cdot q=q$ and $\alpha\cdot\left(\beta\cdot q\right)=\left(\alpha\beta\right)\cdot q$.
\item[(A2)]  A quantity can be multiplied by a quantity, and the quantities form
a commutative monoid with a finite quantity space basis, that is,
a set of quantities $B=\left\{ b_{1},\ldots,b_{n}\right\} $, where
each $b_{i}$ has an inverse $b_{i}^{-1}$, such that each quantity
has a unique expansion $q=\mu\cdot\prod_{i=1}^{n}b_{i}^{k_{i}}$,
where $k_{i}$ are integers.
\item[(A3)]  Multiplication of quantities by numbers and by quantities are related
by the identities $\lambda\cdot\left(pq\right)=\left(\lambda\cdot p\right)q=p\left(\lambda\cdot q\right)$.
\end{enumerate}
It is shown in Parts 2 and 3 how these axioms imply (P1) \textendash{}
(P12) and other algebraic properties of (finitely generated) quantity
spaces. 

\part{Scalable monoids}

Simply stated, a scalable monoid is a monoid whose elements can be
multiplied by numbers and where the identities in (A1) and (A3) hold.
Scalable monoids are formally defined and compared to rings and modules
in Section \ref{s11}, and some basic facts about them are presented
in Section \ref{s12}.

Section \ref{sec:2-7} is concerned with congruences on scalable monoids
and related notions such as orbit classes and other congruence classes,
homomorphisms and quotient algebras. In Section \ref{s16}, direct
and tensor products of scalable monoids are defined. 

Orbit classes and scalable monoids that contain unit elements are
investigated in Sections \ref{sec:2-9} \textendash{} \ref{sec:s11};
addition of elements in the same orbit class is defined, and additive
scalable monoids, ordered scalable monoids and coherent sets of unit
elements are discussed.

A construction-definition of a scalable monoid with a coherent set
of units, similar to that proposed by Carlson \cite{CARL}, is presented
in Section \ref{subsec:19}.

\section{\label{s11}Mathematical background and main definition}

A unital associative algebra $X$ over a (unital, associative) ring
$R$ is equipped with three kinds of operations on $X$: 
\begin{enumerate}
\item \emph{addition} of elements of $X$, a binary operation $+:\left(x,y\right)\mapsto x+y$
on $X$ such that $X$ equipped with $+$ is an abelian group; 
\item \emph{multiplication} of elements of $X$, a binary operation $\ast:\left(x,y\right)\mapsto xy$
on $X$ such that $X$ equipped with $\ast$ is a monoid; 
\item \emph{scalar multiplication} of elements of $X$ by elements of $R,$
a monoid action $\left(\alpha,x\right)\mapsto\alpha\cdot x$ where
the multiplicative monoid of $R$ acts on $X$ so that $1\cdot x=x$
and $\alpha\cdot\left(\beta\cdot x\right)=\alpha\beta\cdot x$ for
all $\alpha,\beta\in R$ and $x\in X$. 
\end{enumerate}
These structures are linked pairwise: 
\begin{enumerate}
\item[(a)] addition and multiplication are linked by the distributive laws $x\left(y+z\right)=xy+xz$
and $\left(x+y\right)z=xz+yz$; 
\item[(b)] addition and scalar multiplication are linked by the distributive
laws\linebreak{}
 $\alpha\cdot\left(x+y\right)=\alpha\cdot x+\alpha\cdot y$ and $\left(\alpha+\beta\right)\cdot x=\alpha\cdot x+\beta\cdot x$; 
\item[(c)] multiplication and scalar multiplication are linked by the bilinearity
laws $\alpha\cdot xy=\left(\alpha\cdot x\right)y$ and $\alpha\cdot xy=x\left(\alpha\cdot y\right)$
\cite{MACLB}. 
\end{enumerate}
Related algebraic structures can be obtained from unital associative
algebras by removing one of the three operations and hence the links
between the removed operation and the two others. Two cases are very
familiar. A\emph{ ring} has only addition and multiplication of elements
of $X$, linked as described in (a). A\emph{ module} has only addition
of elements of $X$ and scalar multiplication of elements of $X$
by elements of $R$, linked as described in (b). The question arises
whether it would be meaningful and useful to define an ``algebra
without an additive group'', with only multiplication of elements
of $X$ and scalar multiplication of elements of $X$ by elements
of $R$, linked as described in (c). 

It would indeed. It turns out that this notion, a ''sibling'' of
rings and modules, referred to as scalable monoids in this article,
makes sense mathematically and is remarkably well suited for modeling
systems of quantities. While (non-extended) numbers are elements of
rings, specifically fields, quantities are elements of scalable monoids,
specifically quantity spaces.
\begin{defn}
\label{thm:def1}Let $R$ be a (unital, associative) ring. A \emph{scalable
monoid} \emph{over} $R$, or $R$\emph{-scaloid}, is a monoid $X$
equipped with a \emph{scaling action} 
\[
\cdot:R\times X\rightarrow X,\qquad\left(\alpha,x\right)\mapsto\alpha\cdot x,
\]
such that for any $\alpha,\beta\in R$ and $x,y\in X$ we have
\begin{enumerate}
\item $1\cdot x=x$, 
\item $\alpha\cdot\left(\beta\cdot x\right)=\alpha\beta\cdot x$, 
\item $\alpha\cdot xy=\left(\alpha\cdot x\right)y=x\left(\alpha\cdot y\right)$. 
\end{enumerate}
\end{defn}
We denote the identity element of $X$ by $1_{\!X}$ or $\mathbf{1}$,
and set $x^{0}=\boldsymbol{1}$ for any $x\in X$. An \emph{invertible}
element of a scalable monoid $X$ is an element $x\in X$ that has
a (necessarily unique) \emph{inverse} $x^{-1}\in X$ such that $xx^{-1}=x^{-1}x=\boldsymbol{1}$.

\section{\label{s12}Some basic facts about scalable monoids}

The following lemma will be used repeatedly.
\begin{lem}
\label{thm:lem1}Let $X$ be a scalable monoid over $R$. For any
$x,y\in X$ and $\alpha,\beta\in R$ we have $\left(\alpha\cdot x\right)\left(\beta\cdot y\right)=\alpha\beta\cdot xy$
and $\alpha\beta\cdot x=\alpha\cdot\left(\beta\cdot x\right)=\beta\cdot\left(\alpha\cdot x\right)=\beta\alpha\cdot x$.
\end{lem}
\begin{proof}
By Definition \ref{thm:def1}, $\left(\alpha\cdot x\right)\left(\beta\cdot y\right)=\alpha\cdot x\left(\beta\cdot y\right)=\alpha\cdot\left(\beta\cdot xy\right)=\alpha\beta\cdot xy.$
Also,
\begin{gather*}
\alpha\beta\cdot x=\alpha\cdot\left(\beta\cdot x\right)=\alpha\cdot\left(\beta\cdot\mathbf{1}x\right)=\alpha\cdot\left(\beta\cdot\mathbf{1}\right)x=\\
\left(\beta\cdot\mathbf{1}\right)\left(\alpha\cdot x\right)=\beta\alpha\cdot\boldsymbol{1}x=\beta\alpha\cdot x=\beta\cdot\left(\alpha\cdot x\right)
\end{gather*}
since $\left(\beta\cdot\mathbf{1}\right)\left(\alpha\cdot x\right)=\beta\alpha\cdot\boldsymbol{1}x$
by the first part of the lemma.
\end{proof}
Let $R$ be a ring, $X$ a monoid. It is easy to verify that the \emph{trivial
scaling action} of $R$ on $X$ defined by $\lambda\cdot x=x$ for
all $\lambda\in R$ and $x\in X$ satisfies conditions (1) \textendash{}
(3) in Definition \ref{thm:def1}, so a monoid equipped with this
function is scalable monoid, namely a \emph{trivially scalable} monoid,
though essentially just a monoid since the operation $\left(\lambda,x\right)\mapsto\lambda\cdot x$
is usually of no interest in this case.

Since every monoid has a unique identity element, the class of all
monoids forms a variety of algebras with a binary operation $\ast:\left(x,y\right)\mapsto xy$,
a nullary operation $\boldsymbol{1}:\left(\right)\mapsto\boldsymbol{1}$
and identities 
\[
x\left(yz\right)=\left(xy\right)z,\qquad\boldsymbol{1}x=x=x\boldsymbol{1}.
\]
The class of all scalable monoids over a fixed ring $R$ is a variety
of algebras in addition equipped with a set of unary operations $\left\{ \omega_{\lambda}\mid\lambda\in R\right\} $,
derived from the external binary operation $\cdot$ in Definition
\ref{thm:def1} by setting $\omega_{\lambda}\left(x\right)=\lambda\cdot x$
for all $\lambda\in R$ and $x\in X$, and with the additional identities
\[
\omega_{1}\left(x\right)=x,\qquad\omega_{\lambda}\left(\omega_{\kappa}\left(x\right)\right)=\omega_{\lambda\kappa}\left(x\right),\qquad\omega_{\lambda}\left(xy\right)=\omega_{\lambda}\left(x\right)\,y=x\,\omega_{\lambda}\left(y\right)\qquad\left(\lambda,\kappa\in R\right),
\]
corresponding to identities (1) \textendash{} (3) in Definition \ref{thm:def1}. 

A scalable monoid is thus a universal algebra 
\[
\left(X,\ast,(\omega{}_{\lambda})_{\lambda\in R},1_{X}\right)
\]
with $X$ as underlying set, here called a \emph{unital magma over}
$R$ or \emph{unital $R$-magma}. The general definitions of direct
products, subalgebras and homomorphisms in the theory of universal
algebras apply. In particular, a subalgebra of a unital $R$-magma
$X$ is a subset $Y$ of $X$ such that $1_{X}\in Y$ and if $x,y\in Y$
and $\lambda\in R$ then $xy,\lambda\cdot x\in Y$. Also, for given
unital $R$-magmas $X$ and $Y$, a unital $R$-magma homomorphism
$\phi:X\rightarrow Y$ is a function such that $\phi\left(xy\right)=\phi\left(x\right)\phi\left(y\right)$,
$\phi\left(\lambda\cdot x\right)=\lambda\cdot\phi\left(x\right)$
and $\phi\left(1_{X}\right)=1_{Y}$ for any $x,y\in X$ and $\lambda\in R$.

By Birkhoff's theorem \cite{BIRK}, varieties are preserved by the
operations of forming subalgebras and homomorphic images. Thus, if
a unital $R$-magma $X$ is a scalable monoid over $R$ then a subalgebra
of $X$ is also a scalable monoid over $R$, and a homomorphic image
of $X$ is also a scalable monoid over $R$.

\section{\label{sec:2-7}Congruences and quotients}

\subsection{\label{subsec:13}On the congruence $\sim$}

In ancient Greek mathematics, the notion of a ratio between magnitudes
only applied to magnitudes of the same kind, so only these could be
commensurable. In this section, we introduce a more radical idea:
quantities are of the same kind \emph{if and only if} they are commensurable.

Let $R\cdot x$ denote the \emph{orbit} of $x\in X$ with regard to
the action $\left(\lambda,x\right)\mapsto\lambda\cdot x$, that is,
the set $\left\{ \lambda\cdot x\mid\lambda\in R\right\} $, and let
$\approx$ denote the relation on $X$ such that $x\approx y$ if
and only if there is some $t\in X$ such that $x,y\in R\cdot t$.
Note that $\approx$ is not an equivalence relation; it is reflexive
since $x\in R\cdot x$ for all $x\in X$ and symmetric by construction
but not transitive, meaning that the orbits of a monoid action may
overlap. 
\begin{defn}
\label{d3.1}Given a scalable monoid $X$ over $R$, let $\sim$ be
the relation on $X$ such that $x\sim y$ if and only if $\alpha\cdot x=\beta\cdot y$
for some $\alpha,\beta\in R.$ 
\end{defn}
Note that $x\sim y$ if and only if $\left(R\cdot x\right)\cap\left(R\cdot y\right)\neq\emptyset$.
We say that $x$ and $y$ are \emph{commensurable} if and only if
$x\sim y$; otherwise $x$ and $y$ are \emph{incommensurable}.
\begin{prop}
\label{s3.1} The relation $\sim$ on a scalable monoid $X$ over
$R$ is an equivalence relation.
\end{prop}
\begin{proof}
The relation $\sim$ is reflexive since $1\cdot x=1\cdot x$ for all
$x\in X$, symmetric by construction, and transitive because if $\alpha\cdot x=\beta\cdot y$
and $\gamma\cdot y=\delta\cdot z$ for some $x,y,z\in X$ and $\alpha,\beta,\gamma,\delta\in R$
then 
\[
\gamma\alpha\cdot x=\gamma\cdot\left(\alpha\cdot x\right)=\gamma\cdot\left(\beta\cdot y\right)=\beta\cdot\left(\gamma\cdot y\right)=\beta\cdot\left(\delta\cdot z\right)=\beta\delta\cdot z,
\]
where $\gamma\alpha,\beta\delta\in R$. 
\end{proof}
An \emph{orbit} \emph{class} $\mathsf{C}$ is an equivalence class
for $\sim$. The orbit class that contains $x$ is denoted $\left[x\right]$,
and $X/{\sim}$ denotes the set $\left\{ \left[x\right]\mid x\in X\right\} $.
\begin{prop}
If $x\sim y$ then $\lambda\cdot x\sim y$, $x\sim\lambda\cdot y$
and $\lambda\cdot x\sim\lambda\cdot y$ for all $\lambda\in R$.
\end{prop}
\begin{proof}
If $x\sim y$ then $\alpha\cdot x=\beta\cdot y$ for some $\alpha,\beta\in R$,
so 
\[
\alpha\lambda\cdot x=\alpha\cdot\left(\lambda\cdot x\right)=\lambda\cdot\left(\alpha\cdot x\right)=\lambda\cdot\left(\beta\cdot y\right)=\beta\cdot\left(\lambda\cdot y\right)=\beta\lambda\cdot y,
\]
where $\alpha\lambda,\beta\lambda\in R$. 
\end{proof}
\begin{cor}
\label{c31}$\lambda\cdot x\sim x$ and $x\sim\lambda\cdot x$ for
all $x\in X$ and $\lambda\in R$.
\end{cor}
\begin{prop}
We have $0\cdot x=0\cdot y$ if and only if $x\sim y$.
\end{prop}
\begin{proof}
If $\alpha\cdot x=\beta\cdot y$ then $0\cdot x=0\alpha\cdot x=0\cdot\left(\alpha\cdot x\right)=0\cdot\left(\beta\cdot y\right)=\left(0\beta\right)\cdot y=0\cdot y$.
\end{proof}
Thus, for every orbit class $\mathsf{C}$ there is a unique $0_{\mathsf{C}}\in\mathsf{C}$
such that $0_{\mathsf{C}}=0\cdot x$ for all $x\in\mathsf{C}$, and
if $0_{\mathsf{A}}=0_{\mathsf{B}}$ then $\mathsf{A}=\mathsf{B}$;
$0_{\mathsf{C}}$ is the \emph{zero element} of $\mathsf{C}$. It
is clear that $\lambda\cdot0_{\mathsf{C}}=0_{\mathsf{C}}$ for all
$\lambda\in R$, and that $0_{\left[x\right]}y=0_{\left[xy\right]}$
and $y0_{\left[x\right]}=0_{\left[yx\right]}$ for all $x,y\in X$.

If $x=\alpha\cdot t$ and $y=\beta\cdot t$ then $\beta\cdot x=\beta\cdot\left(\alpha\cdot t\right)=\alpha\cdot\left(\beta\cdot t\right)=\alpha\cdot y$,
so if $x\approx y$ then $x\sim y$. If $t\in R\cdot x$ then $t,x\in R\cdot x$,
so $t\approx x$, so $t\sim x$, so $t\in\left[x\right]$; hence,
$R\cdot x\subseteq\left[x\right]$ for all $x\in X$. As a consequence,
$\cup_{t\in\left[x\right]}R\cdot t=\left[x\right]$.

It is instructive to relate the present notion of commensurability
to the classical one. We say that $x$ and $y$ are \emph{strongly
commensurable} if and only if $x\approx y$; otherwise, $x$ and $y$
are \emph{weakly incommensurable}. Incommensurability of magnitudes
in the Pythagorean sense obviously corresponds to weak incommensurability. 

We have thus weakened the classical notion of commensurability here,
and this makes it possible to reasonably stipulate that two magnitudes
(elements of a scalable monoid) are of the same kind if and only if
they are commensurable. The deeper significance of the redefinition
of commensurability may be said to be that we have shown how to replace
the intuitive notion of magnitudes of the same kind by the formally
defined notion of commensurable magnitudes.
\begin{prop}
\label{s3.2}Let $X$ be a scalable monoid over $R$. The relation
$\sim$ is a congruence on $X$ with regard to the operations $\left(x,y\right)\mapsto xy$
and $\left(\lambda,x\right)\mapsto\lambda\cdot x$.
\end{prop}
\begin{proof}
For any $x,x',y,y'\in X$ and $\alpha,\alpha',\beta,\beta'\in R$,
we have that if $\alpha\cdot x=\alpha'\cdot x'$ and $\beta\cdot y=\beta'\cdot y'$
then $\left(\alpha\cdot x\right)\left(\beta\cdot y\right)=\left(\alpha'\cdot x'\right)\left(\beta'\cdot y'\right)$,
so $\alpha\beta\cdot xy=\alpha'\beta'\cdot x'y'$ by Lemma \ref{thm:lem1}.
This means that if $x\sim x'$ and $y\sim y'$ then $xy\sim x'y'$.
Also, recall that if $x\sim x'$ then $\lambda\cdot x\sim\lambda\cdot x'$
for any $\lambda\in R$.
\end{proof}
We can thus define a binary operation on $X/{\sim}$ by setting $[x][y]=[xy]$
(so that if $\mathsf{A},\mathsf{B}\in X/{\sim}$, $a\in\mathsf{A}$
and $b\in\mathsf{B}$ then $ab\in\mathsf{AB}\in X/{\sim})$. We can
also set $\lambda\cdot\left[x\right]=\left[\lambda\cdot x\right]$
and $1_{X/{\sim}}=\left[1_{X}\right]$. Given these definitions, the
surjective function $\phi:X\rightarrow X/{\sim}$ given by $\phi\left(x\right)=\left[x\right]$
satisfies the conditions 
\[
\phi\left(xy\right)=\phi\left(x\right)\phi\left(y\right),\quad\phi\left(\lambda\cdot x\right)=\lambda\cdot\phi\left(x\right),\quad\phi\left(1_{X}\right)=1_{X/{\sim}}.
\]
These identities induce a unital $R$-magma structure on $X/{\sim}$,
and by Birkhoff's theorem $X/{\sim}$ is an $R$-scaloid, so $\phi$
is a scalable monoid homomorphism. Thus, Proposition \ref{s3.2},
which is expressed in terms of congruences, leads to Proposition \ref{s3.3-1},
expressed in terms of homomorphisms.
\begin{prop}
\label{s3.3-1}If $X$ is a scalable monoid over $R$ then the quotient
space $X/{\sim}$ is a scalable monoid over $R$, and the function

\[
\phi:X\rightarrow X/{\sim},\qquad x\mapsto\left[x\right],
\]
is a surjective scalable monoid homomorphism.
\end{prop}
In this case, $\lambda\cdot\left[x\right]=\left[\lambda\cdot x\right]=\left[x\right]$
by Corollary \ref{c31}, so we have the following fact.
\begin{prop}
\label{p5-1}If $X$ is a scalable monoid then $X/{\sim}$ is a (trivially
scalable) monoid.
\end{prop}

\subsection{\label{s14}On congruences of the form $\sim_{\!\mathcal{M}}$}

In a monoid we have $x\left(yz\right)=\left(xy\right)z$ and $\boldsymbol{1}x=x=x\boldsymbol{1}$,
so a submonoid $\mathcal{M}$ of a scalable monoid $X$ can act as
a monoid on $X$ by left or right multiplication. In particular, we
can define a monoid action $\left(m,x\right)\mapsto m\star x$ on
a scalable monoid $X$ by setting $m\star x=mx$ for any $m\in\mathcal{M}$
and $x\in X$. For any $x\in X$, the orbit of $x$ with regard to
this action is the right coset $\mathcal{M}x=\left\{ mx\mid m\in\mathcal{M}\right\} $.
Definition \ref{def5} below is analogous to Definition \ref{d3.1},
interpreting left multiplication as a monoid action. 
\begin{defn}
\label{def5}Let $X$ be a scalable monoid and $\mathcal{M}$ a submonoid
of $X$. Then $\sim_{\mathcal{M}}$ is the relation on $X$ such that
$x\sim_{\!\mathcal{M}}y$ if and only if $mx=ny$ for some $m,n\in\mathcal{M}$. 
\end{defn}
\begin{prop}
\label{s3.1-1}If $X$ is a scalable monoid and $\mathcal{M}$ a commutative
submonoid of $X$ then $\sim_{\!\mathcal{M}}$ is an equivalence relation
on $X$. 
\end{prop}
\begin{proof}
The relation $\sim_{\!\mathcal{M}}$ is reflexive since $1_{X}x\sim_{\!\mathcal{M}}1_{X}x$
for all $x\in X$, symmetric by construction, and transitive because
if $mx=ny$ and $m'y=n'z$ for $m,n,m',n'\in\mathcal{M}$ then $m'mx=m'ny=nm'y=nn'z$,
where $m'm,nn'\in\mathcal{M}$.
\end{proof}
We denote the equivalence class $\left\{ t\mid t\sim_{\!\mathcal{M}}x\right\} $
for $\sim_{\!\mathcal{M}}$ by $\left[x\right]_{\mathcal{M}}$, and
the set of equivalence classes $\left\{ \left[x\right]_{\mathcal{M}}\mid x\in X\right\} $
by $X/\mathcal{M}$.

The \emph{center} of a scalable monoid $X$, denoted $Z\left(X\right)$,
is the set of elements of $X$ each of which commutes with all elements
of $X$; clearly, $1_{X}\in Z\left(X\right)$. A \emph{central submonoid}
of a scalable monoid $X$ is a submonoid $\mathcal{M}$ of $X$ such
that $\mathcal{M}\subseteq Z\left(X\right)$. We have the following
corollary of Proposition \ref{s3.1-1}.
\begin{cor}
\label{cor:ny3}If $X$ is a scalable monoid and $\mathcal{M}$ a
central submonoid of $X$ then $\sim_{\!\mathcal{M}}$ is an equivalence
relation on $X$.
\end{cor}
Results analogous to Propositions \ref{s3.2} and \ref{s3.3-1} hold
for central submonoids. 
\begin{prop}
\label{s3.2-1-2}If $X$ is a scalable monoid and $\mathcal{M}$ a
central submonoid of $X$ then the relation $\sim_{\!\mathcal{M}}$
is a congruence on $X$ with regard to the operations $\left(x,y\right)\mapsto xy$
and $\left(\lambda,x\right)\mapsto\lambda\cdot x$.
\end{prop}
\begin{proof}
If $x,x',y,y'\in X$ and $m,n,m',n'\in\mathcal{M}$ then $mx=nx'$
and $m'y=n'y'$ implies $(mx)(m'y)=(nx')(n'y')$ so that $\left(mm'\right)\left(xy\right)=\left(nn'\right)\left(x'y'\right)$.
Hence, if \linebreak{}
$x\sim_{\!\mathcal{M}}x'$ and $y\sim_{\!\mathcal{M}}y'$ then $xy\sim_{\!\mathcal{M}}x'y'$
since $mm',nn'\in\mathcal{M}$. 

Also, if $mx=nx'$ for some $m,n\in M$ then $\lambda\cdot mx=\lambda\cdot nx'$
for all $\lambda\in R$, so $m\left(\lambda\cdot x\right)=n\left(\lambda\cdot x'\right)$.
Hence, if $x\sim_{\!\mathcal{M}}x'$ then $\lambda\cdot x\sim_{\!\mathcal{M}}\lambda\cdot x'$.
\end{proof}
\begin{cor}
\label{c1-1}If $X$ is a commutative scalable monoid and $\mathcal{M}$
a submonoid of $X$ then the relation $\sim_{M}$ is a congruence
on $X$ with regard to the operations $\left(x,y\right)\mapsto xy$
and $\left(\lambda,x\right)\mapsto\lambda\cdot x$.
\end{cor}
We can thus define two operations on $X/\mathcal{M}$ by setting $[x]_{\mathcal{M}}[y]_{\mathcal{M}}=[xy]_{\mathcal{M}}$
and\linebreak{}
$\lambda\cdot\left[x\right]_{\mathcal{M}}=\left[\lambda\cdot x\right]_{\mathcal{M}}$.
We also set $1_{X/\mathcal{M}}=\left[1_{X}\right]_{\mathcal{M}}$.
Given these definitions, the surjective function $\phi_{\mathcal{M}}:X\rightarrow X/\mathcal{M}$
defined by $\phi_{\mathcal{M}}\left(x\right)=\left[x\right]_{\mathcal{M}}$
satisfies the conditions 
\[
\phi_{\mathcal{M}}\left(xy\right)=\phi_{\mathcal{M}}\left(x\right)\phi_{\mathcal{M}}\left(y\right),\quad\phi_{\mathcal{M}}\left(\lambda\cdot x\right)=\lambda\cdot\phi_{\mathcal{M}}\left(x\right),\quad\phi_{\mathcal{M}}\left(1_{X}\right)=1_{X/\mathcal{M}}.
\]
These identities induce a unital $R$-magma structure on $X/\mathcal{M}$,
and by Birkhoff's \linebreak{}
theorem $X/\mathcal{M}$ is an $R$-scaloid, so $\phi_{\mathcal{M}}$
is a scalable monoid homomorphism. \linebreak{}
Proposition \ref{s3.2-1-2} thus corresponds to the following result
about homomorphisms.
\begin{prop}
\label{p8}If $X$ is a scalable monoid over $R$ and $\mathcal{M}$
a central submonoid of $X$ then the quotient space $X/\mathcal{M}$
is a scalable monoid over $R$ and the function 
\[
\phi_{\mathcal{M}}:X\rightarrow X/\mathcal{M},\qquad x\mapsto\left[x\right]_{\mathcal{M}}
\]
is a surjective scalable monoid homomorphism.
\end{prop}

\subsection{On congruences of the form $\sim_{\!\mathscr{M}}$}

Recall that a subalgebra of a scalable monoid $X$ is itself a scalable
monoid, namely, a submonoid $\mathscr{M}$ of $X$ such that $\lambda\cdot x\in\mathscr{M}$
for every $\lambda\in R$ and $x\in\mathscr{M}$; we call $\mathscr{M}$
a \emph{scalable submonoid} of $X$. A \emph{central} scalable submonoid
of $X$ is defined in the same way as a central submonoid of $X$.
\begin{defn}
\label{def5-1}Let $X$ be a scalable monoid and $\mathscr{M}$ a
scalable submonoid of $X$. Then $\sim_{\mathscr{M}}$ is the relation
on $X$ such that $x\sim_{\!\mathcal{\mathscr{M}}}y$ if and only
if $mx=ny$ for some $m,n\in\mathcal{\mathscr{M}}$. 
\end{defn}
The following two results correspond to Proposition \ref{s3.1-1}
and Corollary \ref{cor:ny3} above, and Proposition \ref{s3.1-1-1}
is proved as Proposition \ref{s3.1-1}.
\begin{prop}
\label{s3.1-1-1}If $X$ is a scalable monoid and $\mathcal{\mathcal{\mathscr{M}}}$
a commutative scalable submonoid of $X$ then $\sim_{\!\mathcal{\mathscr{M}}}$
is an equivalence relation on $X$. 
\end{prop}
\begin{cor}
\label{cor:ny3-1}If $X$ is a scalable monoid and $\mathcal{\mathscr{M}}$
a central scalable submonoid of $X$ then $\sim_{\!\mathcal{\mathcal{\mathscr{M}}}}$
is an equivalence relation on $X$.
\end{cor}
The following result is proved as Proposition \ref{s3.2-1-2} above.
\begin{prop}
\label{s3.2-1-2-1}If $X$ is a scalable monoid and $\mathcal{\mathscr{M}}$
a central scalable submonoid of $X$ then the relation $\sim_{\!\mathcal{\mathscr{M}}}$
is a congruence on $X$ with regard to the operations $\left(x,y\right)\mapsto xy$
and $\left(\lambda,x\right)\mapsto\lambda\cdot x$.
\end{prop}
It is clear that we can define $\left[x\right]_{\mathscr{M}}$, $X/\mathscr{M}$,
$\left[x\right]_{\mathscr{M}}\left[y\right]_{\mathscr{M}}$, $\lambda\cdot\left[x\right]_{\mathscr{M}}$
and $1_{X/\mathcal{\mathscr{M}}}$ by modifying the corresponding
definitions in Section \ref{s14}.

Note that for any central scalable submonoid $\mathscr{M}$, we have
$\left[\lambda\cdot x\right]_{\mathscr{M}}=\left[x\right]_{\mathscr{M}}$
since $\boldsymbol{1}\left(\lambda\cdot x\right)=\lambda\cdot x=\lambda\cdot\boldsymbol{1}x=\left(\lambda\cdot\boldsymbol{1}\right)x$,
where $\boldsymbol{1},\lambda\cdot\boldsymbol{1}\in\mathscr{M}$,
so that $\left(\lambda\cdot x\right)\sim_{\!\mathscr{M}}x$ for any
$x\in X$. Hence, $\lambda\cdot\left[x\right]_{\mathscr{M}}=\left[\lambda\cdot x\right]_{\mathscr{M}}=\left[x\right]_{\mathscr{M}}$
for any $\lambda\in R$ and $\left[x\right]_{\mathscr{M}}\in X/\mathscr{M}$,
so while $X/\mathcal{M}$ is a scalable monoid, $X/\mathscr{M}$ is
a (trivially scalable) monoid.

Furthermore, if $\alpha\cdot x=\beta\cdot y$ for some $\alpha,\beta\in R$
then $\left(\alpha\cdot\boldsymbol{1}\right)x=\alpha\cdot\boldsymbol{1}x=\beta\cdot\boldsymbol{1}y=\left(\beta\cdot\boldsymbol{1}\right)y$.
Thus $x\sim y$ implies $x\sim_{\!\mathscr{M}\!}y$ for any central
scalable submonoid $\mathscr{M}$ of $X$ since $\lambda\cdot\boldsymbol{1}\in\mathscr{M}$
for any $\lambda\in R$ and any such $\mathscr{M}$. Conversely, note
that $R\cdot\boldsymbol{1}$ is a central scalable submonoid of $X$
and if $x\sim_{R\,\cdot\boldsymbol{1}}y$ then $\alpha\cdot\boldsymbol{1}x=\left(\alpha\cdot\boldsymbol{1}\right)x=\left(\beta\cdot\boldsymbol{1}\right)y=\beta\cdot\boldsymbol{1}y$
for some $\alpha,\beta\in R$, so $x\sim_{R\,\cdot\boldsymbol{1}}y$
implies $x\sim y$. Thus, $x\sim_{R\,\cdot\boldsymbol{1}}y$ if and
only if $x\sim y$, so $x\sim_{\!\mathscr{M}}y$ generalizes $x\sim y$.

\section{\label{s16}Direct and tensor products of scalable monoids}

Let $X$ and $Y$ be scalable monoids. The \emph{direct product} of
$X$ and $Y$, denoted $X\boxtimes Y$, is the set $X\times Y$ equipped
with the binary operation
\begin{gather*}
\ast:\left(X\times Y\right)\times\left(X\times Y\right)\rightarrow X\times Y,\\
\left(\left\langle x_{1},y_{1}\right\rangle ,\left\langle x_{2},y_{2}\right\rangle \right)\mapsto\left\langle x_{1},y_{1}\right\rangle \left\langle x_{2},y_{2}\right\rangle :=\left\langle x_{1}x_{2},y_{1}y_{2}\right\rangle 
\end{gather*}
and the external binary operation
\[
\cdot:R\times\left(X\times Y\right)\rightarrow X\times Y,\qquad\left(\lambda,\left\langle x,y\right\rangle \right)\mapsto\lambda\cdot\left\langle x,y\right\rangle :=\left\langle \lambda\cdot x,\lambda\cdot y\right\rangle .
\]
Straight-forward calculations show that $X\boxtimes Y$ is a scalable
monoid with $\ast$ as monoid multiplication, $\cdot$ as scaling
action and $\left\langle 1_{X},1_{Y}\right\rangle $ as identity element.

The direct product of scalable monoids is a generic product, applicable
to any universal algebra. Another kind of product, which exploits
the fact that $\left(\lambda\cdot x\right)y=\lambda\cdot xy=x\left(\lambda\cdot y\right)$
in scalable monoids, namely the \emph{tensor product,} turns out to
be more useful in many cases.
\begin{defn}
Given scalable monoids $X$ and $Y$ over $R$, let $\backsim_{\otimes}$
be the binary relation on $X\times Y$ such that $\left(x_{1},y_{1}\right)\backsim_{\otimes}\left(x_{2},y_{2}\right)$
if and only if $\left(\alpha\cdot x_{1},\beta\cdot y_{1}\right)=\left(\beta\cdot x_{2},\alpha\cdot y_{2}\right)$
for some $\alpha,\beta\in R$.
\end{defn}
\begin{prop}
Let $X$ and $Y$ be scalable monoids over $R$. Then $\backsim_{\otimes}$
is an equivalence relation on $X\times Y$. 
\end{prop}
\begin{proof}
$\backsim_{\otimes}$ is reflexive since $\left(1\cdot x,1\cdot y\right)=\left(1\cdot x,1\cdot y\right)$,
and symmetric by construction. If $\left(\alpha\cdot x_{1},\beta\cdot y_{1}\right)=\left(\beta\cdot x_{2},\alpha\cdot y_{2}\right)$
and $\left(\gamma\cdot x_{2},\delta\cdot y_{2}\right)=\left(\delta\cdot x_{3},\gamma\cdot y_{3}\right)$
then 
\begin{gather*}
\left(\gamma\cdot\left(\alpha\cdot x_{1}\right),\delta\cdot\left(\beta\cdot y_{1}\right)\right)=\left(\gamma\cdot\left(\beta\cdot x_{2}\right),\delta\cdot\left(\alpha\cdot y_{2}\right)\right),\\
\left(\beta\cdot\left(\gamma\cdot x_{2}\right),\alpha\cdot\left(\delta\cdot y_{2}\right)\right)=\left(\beta\cdot\left(\delta\cdot x_{3}\right),\alpha\cdot\left(\gamma\cdot y_{3}\right)\right).
\end{gather*}
Thus, we have
\begin{gather*}
\left(\gamma\alpha\cdot x_{1},\delta\beta\cdot y_{1}\right)=\left(\gamma\beta\cdot x_{2},\delta\alpha\cdot y_{2}\right)=\left(\beta\gamma\cdot x_{2},\alpha\delta\cdot y_{2}\right)=\\
\left(\beta\delta\cdot x_{3},\alpha\gamma\cdot y_{3}\right)=\left(\delta\beta\cdot x_{3},\gamma\alpha\cdot y_{3}\right),
\end{gather*}
where $\gamma\alpha,\delta\beta\in R$, so $\backsim_{\otimes}$ is
transitive as well.
\end{proof}
For any $x\in X$ and $y\in Y$, let $x\otimes y$ denote the equivalence
class 
\[
\left\{ \left(s,t\right)\mid\left(s,t\right)\in X\times Y,\left(s,t\right)\backsim_{\otimes}\left(x,y\right)\right\} ,
\]
and let $X\otimes Y$ denote the set $\left\{ x\otimes y\mid x\in X,y\in Y\right\} $.
\begin{prop}
Let $X,Y$ be scalable monoids over $R$, $x\in X$ and $y\in X$.
Then $\left(\lambda\cdot x\right)\otimes y=x\otimes\left(\lambda\cdot y\right)$
for every $\lambda\in R$.
\end{prop}
\begin{proof}
We have $\left(1\cdot\left(\lambda\cdot x\right),\lambda\cdot y\right)=\left(\lambda\cdot x,1\cdot\left(\lambda\cdot y\right)\right)$,
so $\left(\lambda\cdot x,y\right)\backsim_{\otimes}\left(x,\lambda\cdot y\right)$,
meaning that $\left(\lambda\cdot x\right)\otimes y=x\otimes\left(\lambda\cdot y\right)$.
\end{proof}
\begin{prop}
Let $X,Y$ be scalable monoids over $R$, and set $\left(x_{1}\otimes y_{1}\right)\left(x_{2}\otimes y_{2}\right)=x_{1}x_{2}\otimes y_{1}y_{2}$
and $\lambda\cdot x\otimes y=\left(\lambda\cdot x\right)\otimes y$.
Also set $1_{X\otimes Y}=1_{X}\otimes1_{Y}$. With these definitions,
$X\otimes Y$ is a scalable monoid over $R$.
\end{prop}
\begin{proof}
$X\otimes Y$ is a monoid since 
\begin{gather*}
\left(1_{\!X}\otimes1_{\!Y}\right)\left(x\otimes y\right)=1_{\!X}x\otimes1_{\!Y}y=x\otimes y=x1_{\!X}\otimes y1_{\!Y}=\left(x\otimes y\right)\left(1_{\!X}\otimes1_{\!Y}\right),\\
\left(\left(x_{1}\otimes y_{1}\right)\left(x_{2}\otimes y_{2}\right)\right)\left(x_{3}\otimes y_{3}\right)=\left(x_{1}x_{2}\otimes y_{1}y_{2}\right)\left(x_{3}\otimes y_{3}\right)=\left(x_{1}x_{2}\right)x_{3}\otimes\left(y_{1}y_{2}\right)y_{3}=\\
x_{1}\left(x_{2}x_{3}\right)\otimes y_{1}\left(y_{2}y_{3}\right)=\left(x_{1}\otimes y_{1}\right)\left(x_{2}x_{3}\otimes y_{2}y_{3}\right)=\left(x_{1}\otimes y_{1}\right)\left(\left(x_{2}\otimes y_{2}\right)\left(x_{3}\otimes y_{3}\right)\right).
\end{gather*}
Furthermore, 
\begin{gather*}
1\cdot x\otimes y=\left(1\cdot x\right)\otimes y=x\otimes y,\\
\alpha\cdot\left(\beta\cdot x\otimes y\right)=\alpha\cdot\left(\beta\cdot x\right)\otimes y=\left(\alpha\cdot\left(\beta\cdot x\right)\right)\otimes y=\left(\alpha\beta\cdot x\right)\otimes y=\alpha\beta\cdot x\otimes y,\\
\lambda\cdot\left(x_{1}\otimes y_{1}\right)\left(x_{2}\otimes y_{2}\right)=\lambda\cdot x_{1}x_{2}\otimes y_{1}y_{2}=\left(\lambda\cdot x_{1}x_{2}\right)\otimes y_{1}y_{2}=\\
\left(\lambda\cdot x_{1}\right)x_{2}\otimes y_{1}y_{2}=\left(\left(\lambda\cdot x_{1}\right)\otimes y_{1}\right)\left(x_{2}\otimes y_{2}\right)=\left(\lambda\cdot x_{1}\otimes y_{1}\right)\left(x_{2}\otimes y_{2}\right),\\
\lambda\cdot\left(x_{1}\otimes y_{1}\right)\left(x_{2}\otimes y_{2}\right)=\lambda\cdot x_{1}x_{2}\otimes y_{1}y_{2}=x_{1}x_{2}\otimes\left(\lambda\cdot y_{1}y_{2}\right)=\\
x_{1}x_{2}\otimes y_{1}\left(\lambda\cdot y_{2}\right)=\left(x_{1}\otimes y_{1}\right)\left(x_{2}\otimes\left(\lambda\cdot y_{2}\right)\right)=\left(x_{1}\otimes y_{1}\right)\left(\lambda\cdot x_{2}\otimes y_{2}\right),
\end{gather*}
so $X\otimes Y$ is a scalable monoid.
\end{proof}
\begin{cor}
If $X,Y,Z$ are scalable monoids over $R$ then $\left(X\otimes Y\right)\otimes Z$
and\linebreak{}
 $X\otimes\left(Y\otimes Z\right)$ are scalable monoids over $R$.
\end{cor}
\begin{prop}
$\phi:\left(x\otimes y\right)\otimes z\mapsto x\otimes\left(y\otimes z\right)$
is a scalable monoid isomorphism $\left(X\otimes Y\right)\otimes Z\rightarrow X\otimes\left(Y\otimes Z\right)$.
\end{prop}
\begin{proof}
We have 
\begin{gather*}
\phi\left(\left(\left(x_{1}\otimes y_{1}\right)\otimes z_{1}\right)\left(\left(x_{2}\otimes y_{2}\right)\otimes z_{2}\right)\right)=\phi\left(\left(\left(x_{1}\otimes y_{1}\right)\left(x_{2}\otimes y_{2}\right)\right)\otimes z_{1}z_{2}\right)=\\
\phi\left(\left(x_{1}x_{2}\otimes y_{1}y_{2}\right)\otimes z_{1}z_{2}\right)=x_{1}x_{2}\otimes\left(y_{1}y_{2}\otimes z_{1}z_{2}\right)=x_{1}x_{2}\otimes\left(\left(y_{1}\otimes z_{1}\right)\left(y_{2}\otimes z_{2}\right)\right)=\\
\left(x_{1}\otimes\left(y_{1}\otimes z_{1}\right)\right)\left(x_{2}\otimes\left(y_{2}\otimes z_{2}\right)\right)=\phi\left(\left(x_{1}\otimes y_{1}\right)\otimes z_{1}\right)\phi\left(\left(x_{2}\otimes y_{2}\right)\otimes z_{2}\right)
\end{gather*}
 and
\begin{gather*}
\phi\left(\lambda\cdot\left(x\otimes y\right)\otimes z\right)=\phi\left(\left(x\otimes y\right)\otimes\left(\lambda\cdot z\right)\right)=x\otimes\left(y\otimes\left(\lambda\cdot z\right)\right)=\\
x\otimes\left(\lambda\cdot y\otimes z\right)=\lambda\cdot x\otimes\left(y\otimes z\right)=\lambda\cdot\phi\left(\left(x\otimes y\right)\otimes z\right).
\end{gather*}
Also, 
\[
\phi\left(1_{\left(X\otimes Y\right)\otimes Z}\right)=\phi\left(\left(1_{X}\otimes1_{Y}\right)\otimes1_{Z}\right)=1_{X}\otimes\left(1_{Y}\otimes1_{Z}\right)=1_{X\otimes\left(Y\otimes Z\right)}.
\]

Thus, $\phi$ is a scalable monoid homomorphism $\left(X\otimes Y\right)\otimes Z\rightarrow X\otimes\left(Y\otimes Z\right)$,
and similarly $\phi':x\otimes\left(y\otimes z\right)\mapsto\left(x\otimes y\right)\otimes z$
is a scalable monoid homomorphism $X\otimes\left(Y\otimes Z\right)\rightarrow\left(X\otimes Y\right)\otimes Z$
such that $\phi'\circ\phi=\mathrm{Id}_{\left(X\otimes Y\right)\otimes Z}$
and $\phi\circ\phi'=\mathrm{Id}_{X\otimes\left(Y\otimes Z\right)}$,
so $\phi$ is a scalable monoid isomorphism.
\end{proof}

\section{\label{sec:2-9}Orbit classes and unit elements}

\subsection{\label{subsec:91}Orbit classes as modules}

Recall the principle that magnitudes of the same kind can be added
and subtracted, whereas magnitudes of different kinds cannot be combined
by these operations. Also recall the idea that a quantity $q$ can
be represented by a ''unit'' $\left[q\right]$ and a number $\left\{ q\right\} $
specifying ''{[}how many{]} times the {[}unit{]} is to be taken in
order to make up'' the given quantity $q$ \cite{MAXW}. As shown
below, there is a connection between these two notions.

Specifically, it may happen that $R\cdot u\supseteq\left[u\right]$
for some $u\in\left[u\right]$, and if in addition a natural uniqueness
condition is satisfied we may regard $u$ as a unit of measurement
for $\left[u\right]$. If such a unit exists then a sum of magnitudes
in $\left[u\right]$ can be defined by the construction described
in Definition \ref{d3.3}.
\begin{defn}
Let $\mathsf{C}$ be an orbit class in a scalable monoid over $R$.
A \emph{generating element} for $\mathsf{C}$ is some $u\in\mathsf{C}$
such that for every $x\in\mathsf{C}$ there is some $\rho\in R$ such
that $x=\rho\cdot u$. A \emph{unit element for $\mathsf{C}$} is
a generating element $u$ for $\mathsf{C}$ such that if $\rho\cdot u=\rho'\cdot u$
then $\rho=\rho'$. 
\end{defn}
By this definition, if $u$ is a generating element for $\mathsf{C}$
then $R\cdot u\supseteq\mathsf{C}$, but recall that $R\cdot u\subseteq\left[u\right]$,
so actually $R\cdot u=\mathsf{C}=\left[u\right]$, since $u\in\left[u\right]$
and $u\in\mathsf{C}$ so that $\left[u\right]\cap\mathsf{C}\neq\emptyset$.
For any scalable monoid $X$ and $x,u\in X$, $\left[x\right]\subseteq R\cdot u$
implies $x\sim u$ since $1\cdot x=\rho\cdot u$ for some $\rho\in R$,
and $x\sim u$ implies $\left[x\right]=\left[u\right]\supseteq R\cdot u$.

A \emph{trivial} orbit class is an orbit class $\mathsf{C}$ such
that $\mathsf{C}=\left\{ 0_{\mathsf{C}}\right\} $. As $\lambda\cdot0_{\mathsf{C}}=0_{\mathsf{C}}$
for any $\lambda\in R$, a zero element cannot be a unit element for
a non-trivial orbit class. The existence of a non-zero unit for an
orbit class in a scalable monoid over $R$ has implications for $R$.
\begin{prop}
\label{p91}If there exists a unit element for some non-trivial orbit
class in a scalable monoid over $R$ then $R$ is non-trivial and
commutative.
\end{prop}
\begin{proof}
Let $u$ be the unit element and $0\cdot u$ the zero element. Then
$0\neq1$ since $0\cdot u\neq1\cdot u$. We also have $\alpha\beta\cdot u=\beta\alpha\cdot u$
for any $\alpha,\beta\in R$ by Lemma \ref{thm:lem1}.
\end{proof}
We now turn to an argument that leads to Proposition \ref{p5}.
\begin{prop}
Let $\mathsf{C}$ be an orbit class in a scalable monoid over $R$.
If $u$ and $u'$ are unit elements for $\mathsf{C}$, $\rho\cdot u=\rho'\cdot u'$
and $\sigma\cdot u=\sigma'\cdot u'$ then $\left(\rho+\sigma\right)\cdot u=\left(\rho'+\sigma'\right)\cdot u'$. 
\end{prop}
\begin{proof}
As $u'\in\mathsf{C}$, there is a unique $\tau\in R$ such that $u'=\tau\cdot u$.
Thus, 
\[
\left(\rho'+\sigma'\right)\cdot u'=\left(\rho'+\sigma'\right)\cdot\left(\tau\cdot u\right)=\left(\rho'+\sigma'\right)\tau\cdot u=\left(\rho'\tau+\sigma'\tau\right)\cdot u=\left(\rho+\sigma\right)\cdot u
\]
since $\rho\cdot u=\rho'\cdot u'=\rho'\cdot\left(\tau\cdot u\right)=\rho'\tau\cdot u$
and $\sigma\cdot u=\sigma'\cdot u'=\sigma'\cdot\left(\tau\cdot u\right)=\sigma'\tau\cdot u$,
so that $\rho=\rho'\tau$ and $\sigma=\sigma'\tau$.
\end{proof}
Hence, the sum of two elements of a scalable monoid can be defined
as follows.
\begin{defn}
\label{d3.3}Let $X$ be a scalable monoid over $R$, and let $u$
be a unit element for an orbit class $\mathsf{C}$. If $x=\rho\cdot u$
and $y=\sigma\cdot u$, we set 
\[
x+y=\left(\rho+\sigma\right)\cdot u.
\]
\end{defn}
Thus, if $x,y\in\mathsf{C}$ then $x+y=\left(\rho+\sigma\right)\cdot u\in R\cdot u=\mathsf{C}$,
and if $x\in\mathsf{C}$ then $\lambda\cdot x=\lambda\cdot\left(\rho\cdot u\right)=\lambda\rho\cdot u\in R\cdot u=\mathsf{C}$.
We note that the sum $x+y$ is given by Definition \ref{d3.3} if
and only if $x$ and $y$ are commensurable and their orbit class
has a unit element. This fact motivates that the notion of magnitudes
of the same kind is replaced by that of commensurable magnitudes (see
Section \ref{subsec:13}).

It follows immediately from Definition \ref{d3.3} that 
\[
\left(x+y\right)+z=x+\left(y+z\right),\qquad x+y=y+x
\]
for all $x,y,z\in\mathsf{C}$, and that 
\[
x+0_{\mathsf{C}}=0_{\mathsf{C}}+x=x
\]
 for any $x\in\mathsf{C}$ since $0_{\mathsf{C}}=0\cdot u$. 

If $x=\rho\cdot u$ so that $\lambda\cdot x=\lambda\rho\cdot u$ and
$\kappa\cdot x=\kappa\rho\cdot u$ then 
\begin{gather*}
\left(\lambda+\kappa\right)\cdot x=\left(\lambda+\kappa\right)\cdot\left(\rho\cdot u\right)=\left(\lambda+\kappa\right)\rho\cdot u=\left(\lambda\rho+\kappa\rho\right)\cdot u=\lambda\cdot x+\kappa\cdot x,
\end{gather*}
and if $x=\rho\cdot u$ and $y=\sigma\cdot u$ so that $\lambda\cdot x=\lambda\rho\cdot u$
and $\lambda\cdot y=\lambda\sigma\cdot u$ then 
\begin{gather*}
\lambda\cdot\left(x+y\right)=\lambda\cdot\left(\left(\rho+\sigma\right)\cdot u\right)=\lambda\left(\rho+\sigma\right)\cdot u=\left(\lambda\rho+\lambda\sigma\right)\cdot u=\lambda\cdot x+\lambda\cdot y.
\end{gather*}

A unital ring $R$ has a unique additive inverse $-1$ of $1\in R$,
and we set 
\[
-x=\left(-1\right)\cdot x
\]
for all $x\in X$. If $\mathsf{C}$ has a unit  element $u$ and $x=\rho\cdot u$
for some $\rho\in R$ then
\[
x+\left(-x\right)=-x+x=0_{\mathsf{C}}
\]
since $x+\left(-x\right)=\rho\cdot u+\left(-\rho\right)\cdot u=\left(\rho+\left(-\rho\right)\right)\cdot u=0\cdot u$
and $-x+x=\left(-\rho\right)\cdot u+\rho\cdot u=\left(-\rho+\rho\right)\cdot u=0\cdot u$,
using the fact that $-x=\left(-1\right)\cdot\left(\rho\cdot u\right)=\left(-\rho\right)\cdot u$. 

As usual, we may write $x+\left(-y\right)$ as $x-y$, and thus $x+\left(-x\right)$
as $x-x$.

We conclude that an orbit class $\mathsf{C}\in X/{\sim}$ with a unit
element is a module with addition in $\mathsf{C}$ given by Definition
\ref{d3.3}, and scalar multiplication in $\mathsf{C}$ inherited
from the scalar multiplication in $X$. A non-trivial orbit class
with a unit element is an \emph{orbit class module} with a simple
structure.
\begin{prop}
\label{p5}Let $X$ be a scalable monoid over $R$. If $\mathsf{C}\in X/{\sim}$
is a non-trivial orbit class with a unit element then $\mathsf{C}$
is a free module of rank 1 over $R$.
\end{prop}
\begin{proof}
We have shown that $\mathsf{C}$ is a module. Also, if $u$ is a unit
element for $\mathsf{C}$ then $\left\{ u\right\} $ is a basis for
$\mathsf{C}$, and $R$ is a non-trivial commutative ring by Proposition
\ref{p91}, so $R$ has the invariant basis number property \cite{RICH}.
\end{proof}
Thus, if every orbit class $\mathsf{C}\in X/{\sim}$ contains a non-zero
unit element for $\mathsf{C}$ then $X$ is the union of disjoint
isomorphic free modules of rank 1. This fact may be compared to Quade's
and Raposo's definitions of quantity spaces \cite{QUAD,RAP}.

Recall that identities corresponding to $\left(\lambda+\kappa\right)\cdot x=\lambda\cdot x+\kappa\cdot x,$
$\lambda\cdot\left(x+y\right)=\lambda\cdot x+\lambda\cdot y$ and
$\lambda\cdot\left(\kappa\!\cdot\!x\right)=\lambda\kappa\cdot x$
were proved in Propositions 1 \textendash{} 3 in Book V of the \emph{Elements,}
so rudiments of Proposition \ref{p5} were present already in the
Greek theory of magnitudes.

\subsection{Orbit classes as ordered modules}

A \emph{total order} on a set $S$ is a binary relation $\leq$ such
that for all $x,y,z\in S$ we have that
\begin{enumerate}
\item $x\leq y$ or $y\leq x$\emph{;}
\item if $x\leq y$ and $y\leq x$ then $x=y$\emph{;}
\item if $x\leq y$ and $y\leq z$ then $x\leq z$.
\end{enumerate}
An \emph{order} on a ring $R$ is a total order $\leq$ on $R$ such
that
\begin{enumerate}
\item if $x\le y$ then $x+z\leq y+z$\emph{;}
\item if $0\leq x$ and $0\leq y$ then $0\leq xy$.
\end{enumerate}
An \emph{ordered ring} is, of course, a ring equipped with an order
on a ring. Well-known facts about inequalities such as $0\leq1$ and
if $x\leq y$ and $x'\leq y'$ then $x+x'\leq y+y'$ can be derived
from the definition of an ordered ring.
\begin{prop}
Let $\mathsf{C}$ be an orbit class module over an ordered ring $R$,
$\mathrm{U}\left(\mathsf{C}\right)$ the set of all unit elements
for $\mathsf{C}$. Also let $\equiv_{\mathsf{C}}$ be the relation
on $\mathrm{U}\left(\mathsf{C}\right)$ defined by $u\equiv_{\mathsf{C}}v$
if and only if $u=\rho\cdot v$ for some $\rho\in R$ such that $0\leq\rho$.
Then $\equiv_{\mathsf{C}}$ is an equivalence relation on $\mathrm{U}\left(\mathsf{C}\right)$.
\end{prop}
\begin{proof}
The relation $\equiv{}_{\mathsf{C}}$ is reflexive since $u=1\cdot u$,
and transitive since if $u=\rho\cdot v$ and $v=\sigma\cdot w$, where
$0\leq\rho,\sigma$, then $u=\rho\sigma\cdot w$, where $0\leq\rho\sigma$.
Also note that if $u\equiv_{\mathsf{C}}v$, so that $u=\rho\cdot v$,
$v=\tau\cdot u$ and $0\leq\rho$, then $1\cdot u=\rho\cdot v=\rho\tau\cdot u$
and $1\cdot v=\tau\cdot u=\tau\rho\cdot v$, so $\rho\tau=\tau\rho=1$,
so $0\leq\tau$ since $0\leq\rho$. Thus, $\tau\cdot u=\tau\rho\cdot v=v$,
where $0\leq\tau$, so $\equiv_{\mathsf{C}}$ is symmetric as well.
\end{proof}
\begin{prop}
\label{p13}Let $\mathsf{C}$ be an orbit class module over an ordered
ring $R$, and let $u,v$ be unit elements for $\mathsf{C}$ such
that $u\equiv_{\mathsf{C}}v$. For any $x\in\mathsf{C}$, if $x=\rho\cdot u=\sigma\cdot v$
for some $\rho,\sigma\in R$ and $0\leq\rho$ then $0\leq\sigma$.
\end{prop}
\begin{proof}
There is some $\tau\in R$ such that $u=\tau\cdot v$ and $0\leq\tau$.
Thus, $\sigma\cdot v=x=\rho\cdot u=\rho\cdot\left(\tau\cdot v\right)=\rho\tau\cdot v$,
so $\sigma=\rho\tau$, so $0\leq\sigma$ since $0\leq\rho,\tau$.
\end{proof}
\begin{defn}
A set of unit elements $\mathrm{U}\left(\mathsf{C}\right)$ such that
if $u,v\in\mathrm{U}\left(\mathsf{C}\right)$ then $u\equiv_{\mathsf{C}}v$
is said to be \emph{consistent}.
\end{defn}
\begin{defn}
Let $\mathsf{C}$ be an orbit class module over an ordered ring $R$.
An \emph{order} on $\mathsf{C}$ is a total order on $\mathsf{C}$
as a set such that, for all $x,y,z\in\mathsf{C}$ and $\lambda\in R$,
\begin{enumerate}
\item if $x\le_{\mathsf{C}}y$ then $x+z\leq_{\mathsf{C}}y+z$\emph{;}
\item if $x\le_{\mathsf{C}}y$ and $0\leq\lambda$ then $\lambda\cdot x\leq_{\mathsf{C}}\lambda\cdot y$.
\end{enumerate}
\end{defn}
\begin{prop}
\label{prop:p10}Let $\mathsf{C}$ be an orbit class module over an
ordered ring, $\mathrm{U}\left(\mathsf{C}\right)$ a consistent set
of unit elements for $\mathsf{C}$. Then the relation $\leq_{\mathsf{C}}$
on $\mathsf{C}$ defined by $x\leq_{\mathsf{C}}y$ if and only if
$y-x=\rho\cdot u$ for some $\rho\in R$ such that $0\leq\rho$ is
an order on $\mathsf{C}$.
\end{prop}
\begin{proof}
By Proposition \ref{p13} the relation $\leq_{\mathsf{C}}$ does not
depend on a choice of unit element in $\mathsf{C}$. We first show
that $\leq_{\mathsf{C}}$ is a total order on $\mathsf{C}$. If $x,y\in\mathsf{C}$
then $y-x=\rho\cdot u$ for some $\rho\in R,u\in U\cap\mathsf{C}$.
Thus, if $0\leq\rho$ then $x\leq_{\mathsf{C}}y$; if $\rho\leq0$
then $0\leq\left(-\rho\right)$ and $x-y=-\left(y-x\right)=-\left(\rho\cdot u\right)=\left(-1\right)\cdot\left(\rho\cdot u\right)=\left(-\rho\right)\cdot u$,
so $y\leq_{\mathsf{C}}x$. 

If $x\leq_{\mathsf{C}}y$ and $y\leq_{\mathsf{C}}x$ then $0\leq\rho$
and $\rho\leq0$, so $\rho=0$, so $y-x=0_{\mathsf{C}},$ so $x=y$.
Also, if $x,y,z\in\mathsf{C}$, $y-x=\rho\cdot u$ and $z-y=\sigma\cdot u$
then $z-x=\left(y-x\right)+\left(z-y\right)=\left(\rho+\sigma\right)\cdot u$,
so if $x\leq_{\mathsf{C}}y$ and $y\leq_{\mathsf{C}}z$ so that $0\leq\rho,\sigma$
then $x\leq_{\mathsf{C}}z$ since $0\leq\rho+\sigma$. Similarly,
if $x\leq_{\mathsf{C}}y$, meaning that $x=\rho\cdot u$ and $y=\sigma\cdot u$
for some $\rho,\sigma\in R$ such that $0\leq\sigma-\rho$, then $0\leq\lambda$
implies $0\leq\lambda\left(\sigma-\rho\right)=\lambda\sigma-\lambda\rho$,
so $\lambda\cdot x\leq_{\mathsf{C}}\lambda\cdot y$ since $\lambda\cdot y-\lambda\cdot x=\lambda\cdot\left(\sigma\cdot u\right)-\lambda\cdot\left(\rho\cdot u\right)=\left(\lambda\sigma-\lambda\rho\right)\cdot u$.
Also, if $x\le_{\mathsf{C}}y$ then $x+z\leq_{\mathsf{C}}y+z$ since
$\left(y+z\right)-\left(x+z\right)=y-x=\left(\sigma-\rho\right)\cdot u$.
\end{proof}
We call an orbit class module $\mathsf{C}$ equipped with an order
on $\mathsf{C}$ an \emph{ordered orbit class module}.

\section{\label{s25}Scalable monoids and sets of unit elements }
\begin{defn}
\label{d8}A \emph{dense} set of elements of a scalable monoid $X$
is a set $U$ of elements of $X$ such that for every $x\in X$ there
is some $u\in U$ such that $u\sim x$. A \emph{sparse} set of elements
of $X$ is a set $U$ of elements of $X$ such that $u\sim v$ implies
$u=v$ for any $u,v\in U$. A \emph{closed} set of elements of $X$
is a set $U$ of elements of $X$ such that if $u,v\in U$ then $uv\in U$.
A \emph{set of unit elements} of a scalable monoid $X$ is a set of
elements of $X$ each of which is a unit element for some $\mathsf{C}\in X/{\sim}$. 

We call a dense sparse set of unit elements of $X$ a \emph{system
of unit elements }for $X$, and a sparse set of unit elements of $X$
a \emph{partial} \emph{system of unit elements }for $X$.
\end{defn}

\subsection{Additive scalable monoids}
\begin{defn}
\label{d9}An \emph{additive scalable monoid} is a scalable monoid
$X$ such that for every $\mathsf{C}\in X/{\sim}$ there is a binary
operation
\[
+:\mathsf{C}\times\mathsf{C}\rightarrow\mathsf{C},\qquad\left(x,y\right)\mapsto x+y
\]
such that $\mathsf{C}$ equipped with $+$ is an abelian group and
\[
x\left(y+z\right)=xy+xz,\qquad\left(y+z\right)x=yx+yz
\]
 for all $x\in\mathsf{A}$ and $y,z\in\mathsf{B}$ for all $\mathsf{A},\mathsf{B}\in X/{\sim}$.
\end{defn}
\begin{prop}
\label{p11}If a scalable monoid $X$ is equipped with a dense closed
set of unit elements $U$ then $X$ is an additive scalable monoid.
\end{prop}
\begin{proof}
By Proposition \ref{p5}, each $\mathsf{C}\in X/{\sim}$ is a module
since $U$ is dense in $X$.

For all $x\in\mathsf{A}$ and $y,z\in\mathsf{B}$ there are $u,v\in U$
such that $\left[x\right]=\left[u\right]$ and $\left[y\right]=\left[z\right]=\left[v\right]$
since $U$ is dense in $X$, so $x=\rho\cdot u$, $y=\sigma\cdot v$
and $z=\tau\cdot v$ for some $\rho,\sigma,\tau\in R$, so $xy=\rho\sigma\cdot uv,$
$xz=\rho\tau\cdot uv$, $yx=\sigma\rho\cdot vu$ and $zx=\tau\rho\cdot vu$,
so
\begin{gather*}
x\left(y+z\right)=\left(\rho\cdot u\right)\left(\left(\sigma+\tau\right)\cdot v\right)=\rho\left(\sigma+\tau\right)\cdot uv=\left(\rho\sigma+\rho\tau\right)\cdot uv=xy+xz,\\
\left(y+z\right)x=\left(\left(\sigma+\tau\right)\cdot v\right)\left(\rho\cdot u\right)=\left(\sigma+\tau\right)\rho\cdot vu=\left(\sigma\rho+\tau\rho\right)\cdot vu=yx+zx,
\end{gather*}
using the fact that $uv$ and $vu$ are unit elements since $U$ is
closed.
\end{proof}

\subsection{Ordered scalable monoids}
\begin{defn}
\label{def:d14}Let $X$ be a scalable monoid $X$ over an ordered
ring $R$, equipped with a dense set of unit elements. An \emph{order}
on $X$ is a relation $\leq$ on $X$ such that:
\end{defn}
\begin{enumerate}
\item if $x\leq y$ then $x\sim y$ so that $x,y\in\mathsf{C}$ for some
$\mathsf{C}\in X/{\sim}$;
\item the relation $\leq\cap\,\left(\mathsf{C}\times\mathsf{C}\right)$
is an order on $\mathsf{C}$ for every $\mathsf{C}\in X/{\sim}$ ;
\item if $0_{\mathsf{A}}\leq\cap\,\left(\mathsf{A}\times\mathsf{A}\right)\,x$
and $0_{\mathsf{B}}\leq\cap\,\left(\mathsf{B}\times\mathsf{B}\right)\,y$
then $0_{\mathsf{A}\mathsf{B}}\leq\cap\,\left(\mathsf{AB}\times\mathsf{AB}\right)\,xy$
for any $x\in\mathsf{A}$ and $y\in\mathsf{B}$.
\end{enumerate}
If we denote the relation $\leq\cap\,\left(\mathsf{C}\times\mathsf{C}\right)$
by $\leq_{\mathsf{C}}$ then (3) takes a more transparent form:
\begin{enumerate}
\item[(3')]  if $0_{\mathsf{A}}\leq_{\mathsf{A}}x$ and $0_{\mathsf{B}}\leq_{\mathsf{B}}y$
then $0_{\mathsf{A}\mathsf{B}}\leq_{\mathsf{A}\mathsf{B}}xy$ for
any $x\in\mathsf{A}$ and $y\in\mathsf{B}$.
\end{enumerate}
An \emph{ordered scalable monoid} $X$ is a scalable monoid equipped
with an order on $X$.
\begin{defn}
Let $X$ be a scalable monoid over an ordered ring $R$. A \emph{dense
consistent} set of unit elements for $X$ is a dense set $U$ of unit
elements of $X$ such that $U\cap\mathsf{C}$ is a consistent set
of unit elements for every $\mathsf{C}\in X/{\sim}$.
\end{defn}
\begin{prop}
\label{p15}Let $X$ be a scalable monoid over an ordered ring $R$,
equipped with a dense consistent closed set $U$ of unit elements
of $X$. Then there is a unique order $\leq$ on $X$ such that $x\leq_{\mathsf{C}}y$
if and only if $x\sim y$ and $y-x=\rho\cdot u$ for some $\rho\in R$
such that $0\leq\rho$ and some $u\in U\cap\mathsf{C}$.
\end{prop}
\begin{proof}
In view of Proposition \ref{prop:p10}, it suffices to show (3) in
Definition \ref{def:d14}. By Proposition \ref{p13} the relation
$\leq_{\mathsf{C}}$ does not depend on a choice of unit element in
$\mathsf{C}$. If $0_{\mathsf{A}}\leq_{\mathsf{A}}x$ and $0_{\mathsf{B}}\leq_{\mathsf{B}}y$,
meaning that $x=\rho\cdot u$ and $y=\sigma\cdot v$ for some $\rho,\sigma\in R$,
where $0\leq\rho,\sigma$ and $u\in\mathsf{A}$, $v\in\mathsf{B}$,
then $xy=\left(\rho\cdot u\right)\left(\sigma\cdot v\right)=\rho\sigma\cdot uv$,
where $0\leq\rho\sigma$ and $uv\in AB$, so $0_{\mathsf{A}\mathsf{B}}\leq_{\mathsf{A}\mathsf{B}}xy$. 
\end{proof}
We can thus impose an order on a scalable monoid $X$ over an ordered
ring by specifying a dense consistent closed set of unit elements
of $X$. Recall that by Proposition \ref{p11} a scalable monoid with
a dense closed set of unit elements is an additive scalable monoid
as well.

\section{\label{sec:s11}On coherent systems of unit elements}

Let us now bring together some notions from Sections \ref{sec:2-7}
and \ref{s25}.
\begin{defn}
\label{d12}A \emph{coherent }system of unit elements for $X$ is
a submonoid of $X$ which is a system of unit elements for $X$.
\end{defn}
Note that a coherent system of unit elements is dense, sparse and
closed by definition, and trivially consistent.
\begin{prop}
\label{p9}Let $X$ be a scalable monoid, $\mathcal{U}$ a coherent
system of unit\linebreak{}
 elements for $X$, and $\mathcal{V}\subseteq\mathcal{U}$ a central
submonoid of $X$. Then $X/\mathcal{V}$ is a scalable monoid, $\left[v\right]_{\mathcal{V}}=\left[\boldsymbol{1}\right]_{\mathcal{V}}$
for any $v\in\mathcal{V}$, and $\overline{\mathcal{U}}=\left\{ \left[u\right]_{\mathcal{V}}\mid u\in\mathcal{U}\right\} $
is a coherent system of unit elements for $X/\mathcal{V}$. 
\end{prop}
\begin{proof}
As $\mathcal{V}$ is a central submonoid of $X$, $X/\mathcal{V}$
is a scalable monoid, and if $v\in\mathcal{V}$ then $v\sim_{_{\mathcal{V}}\!}\boldsymbol{1}$
since $\boldsymbol{1}v=v\boldsymbol{1}$ and $\boldsymbol{1}\in\mathcal{V}$,
so $\left[v\right]_{\mathcal{V}}=\left[\boldsymbol{1}\right]_{\mathcal{V}}$. 

If $\overline{u},\overline{u}'\in\overline{\mathcal{U}}$ then $\overline{u}=\left[u\right]_{\mathcal{V}}$
and $\overline{u}'=\left[u'\right]_{\mathcal{V}}$ for some $u,u'\in\mathcal{U}$,
and $uu'\in\mathcal{U}$ since $\mathcal{U}$ is a monoid, so $\overline{u}\,\overline{u}'=\left[uu'\right]_{\mathcal{V}}\in\overline{\mathcal{U}}$.
We also have 
\[
\left[\boldsymbol{1}\right]_{\mathcal{V}}\left[u\right]_{\mathcal{V}}=\left[\boldsymbol{1}u\right]_{\mathcal{V}}=\left[u\right]_{\mathcal{V}}=\left[u\boldsymbol{1}\right]_{\mathcal{V}}=\left[u\right]_{\mathcal{V}}\left[\boldsymbol{1}\right]_{\mathcal{V}}
\]
and 
\[
\left(\left[u\right]\left[u'\right]\right)\left[u''\right]=\left[\left(uu'\right)u''\right]=\left[u\left(u'u''\right)\right]=\left[u\right]\left(\left[u'\right]\left[u''\right]\right),
\]
so $\overline{\mathcal{U}}$ is a monoid. Thus, $\overline{\mathcal{U}}$
is a submonoid of $X/\mathcal{V}$.

For any $\left[x\right]_{\mathcal{V}}\in X/\mathcal{V}$ there is
some $u\in U$ and $\rho\in R$ such that $\left[x\right]_{\mathcal{V}}=\left[\rho\cdot u\right]_{\mathcal{V}}=\rho\cdot\left[u\right]_{\mathcal{V}}$;
thus also $\left[x\right]_{\mathcal{V}}\sim\left[u\right]_{\mathcal{V}}$
since $1\cdot\left[x\right]_{\mathcal{V}}=\left[1\cdot x\right]_{\mathcal{V}}=\left[x\right]_{\mathcal{V}}$.
If $\left[x\right]_{\mathcal{V}}=\rho\cdot\left[u\right]_{\mathcal{V}}=\sigma\cdot\left[u\right]_{\mathcal{V}}$
for some $\rho,\sigma\in R$ then $\left[\rho\cdot u\right]_{\mathcal{V}}=\left[\sigma\cdot u\right]_{\mathcal{V}}$,
so $\rho\cdot u\sim_{\mathcal{V}}\sigma\cdot u$, so $v\left(\rho\cdot u\right)=v'\left(\sigma\cdot u\right)$
where $v,v'\in\mathcal{V}$, so $\rho\cdot vu=\sigma\cdot v'u$ where
$vu,v'u\in\mathcal{U}$, so $vu\sim v'u$, so $vu=v'u$ since $\mathcal{U}$
is sparse, so $\rho=\sigma$. Thus, $\overline{\mathcal{U}}$ is a
dense set of unit elements of $X/\mathcal{V}$.

Also, if $u,u'\in\mathcal{U}$ and $\left[u\right]_{\mathcal{V}}\sim\left[u'\right]_{\mathcal{V}}$
then $\rho\cdot\left[u\right]_{\mathcal{V}}=\sigma\cdot\left[u'\right]_{\mathcal{V}}$,
so $\left[\rho\cdot u\right]_{\mathcal{V}}=\left[\sigma\cdot u'\right]_{\mathcal{V}}$,
so $\rho\cdot u\sim_{\mathcal{V}}\sigma\cdot u'$, so $v\left(\rho\cdot u\right)=v'\left(\sigma\cdot u'\right)$
for some $v,v'\in\mathcal{V}$, so $\rho\cdot vu=\sigma\cdot v'u'$,
so $vu\sim v'u'$ where $vu,v'u'\in\mathcal{U}$, so $vu=v'u'$, so
$\left[vu\right]_{\mathcal{V}}=\left[v'u'\right]_{\mathcal{V}}$,
so $\left[v\right]_{\mathcal{V}}\left[u\right]_{\mathcal{V}}=\left[v'\right]_{\mathcal{V}}\left[u'\right]_{\mathcal{V}}$,
so $\left[u\right]_{\mathcal{V}}=\left[u'\right]_{\mathcal{V}}$.
Thus, $\overline{\mathcal{U}}$ is a sparse set of unit elements of
$X/\mathcal{V}$.
\end{proof}
Note that $\mathcal{V}$ is a \emph{partial} coherent system of unit
elements of $X$, that is, a submonoid of $X$ that is a partial system
of unit elements of $X$. For any $v,v'\in\mathcal{V}$ and any $\lambda\in R$,
$\lambda\cdot v$ and $\lambda\cdot v'$ in $X$ correspond to the
same element $\left[\lambda\cdot\boldsymbol{1}\right]_{\mathcal{V}}$
of $X/\mathcal{V}$; more generally, for any $v,v'\in\mathcal{V}$,
any $u,u'\in\mathcal{U}$ and any $\lambda\in R$, $\lambda\cdot uvu'$
and $\lambda\cdot uv'u'$ in $X$ correspond to the element $\left[\lambda\cdot uu'\right]_{\mathcal{V}}$
of $X/\mathcal{V}$.

A typical application of Proposition \ref{p9} in physics is described
by Raposo \cite{RAP}:
\begin{quotation}
{\footnotesize{}The mechanism of taking quotients is the algebraic
tool underlying what is common practice in physics of choosing \textquotedblleft systems
of units'' such that some specified universal constants become dimensionless
and take on the numerical value 1. {[}...{]} But it has to be remarked
that the mechanism goes beyond a change of system of units; it is
indeed a change of space of quantities. }{\footnotesize \par}
\end{quotation}

\section{\label{subsec:19} Ring-monoids as scalable monoids}
\begin{defn}
\label{d13}Let $R$ be a ring and $M$ a monoid. A \emph{ring-monoid}
$R\boxtimes M$ is a set $R\times M$ equipped with a binary operation
\[
\ast:\left(R\times M\right)\times\left(R\times M\right)\rightarrow R\times M,\qquad\left(\left\langle \alpha,x\right\rangle ,\left\langle \beta,y\right\rangle \right)\mapsto\left\langle \alpha,x\right\rangle \left\langle \beta,y\right\rangle :=\left\langle \alpha\beta,xy\right\rangle 
\]
and an external binary operation
\[
\cdot:R\times\left(R\times M\right)\rightarrow R\times M,\qquad\left(\lambda,\left\langle \alpha,x\right\rangle \right)\mapsto\lambda\cdot\left\langle \alpha,x\right\rangle :=\left\langle \lambda\alpha,x\right\rangle .
\]
\end{defn}
\begin{prop}
\label{p17}Let $R\boxtimes M$ be a ring-monoid. If $R$ is a commutative
ring, then $R\boxtimes M$ is a scalable monoid over $R$.
\end{prop}
\begin{proof}
We have
\begin{gather*}
\left(\left\langle \alpha,x\right\rangle \left\langle \beta,y\right\rangle \right)\left\langle \gamma,z\right\rangle =\left\langle \left(\alpha\beta\right)\gamma,\left(xy\right)z\right\rangle =\left\langle \alpha\left(\beta\gamma\right),x\left(yz\right)\right\rangle =\left\langle \alpha,x\right\rangle \left(\left\langle \beta,y\right\rangle \left\langle \gamma,z\right\rangle \right),\\
\left\langle 1,\boldsymbol{1}\right\rangle \left\langle \alpha,x\right\rangle =\left\langle \alpha,x\right\rangle =\left\langle \alpha,x\right\rangle \left\langle 1,\boldsymbol{1}\right\rangle 
\end{gather*}
for any $\alpha,\beta,\gamma\in R$ and $x,y,z\in M$, so $R\boxtimes M$
is a monoid with $\left\langle 1,\boldsymbol{1}\right\rangle $ as
identity element. Furthermore,
\begin{gather*}
1\cdot\left\langle \alpha,x\right\rangle =\left\langle 1\alpha,x\right\rangle =\left\langle \alpha,x\right\rangle ,\\
\lambda\cdot\left(\kappa\cdot\left\langle \alpha,x\right\rangle \right)=\lambda\cdot\left\langle \kappa\alpha,x\right\rangle =\left\langle \lambda\left(\kappa\alpha\right),x\right\rangle =\left\langle \left(\lambda\kappa\right)\alpha,x\right\rangle =\lambda\kappa\cdot\left\langle \alpha,x\right\rangle ,\\
\lambda\cdot\left\langle \alpha,x\right\rangle \left\langle \beta,y\right\rangle =\lambda\cdot\left\langle \alpha\beta,xy\right\rangle =\left\langle \lambda\left(\alpha\beta\right),xy\right\rangle =\\
\left\langle \left(\lambda\alpha\right)\beta,xy\right\rangle =\left\langle \lambda\alpha,x\right\rangle \left\langle \beta,y\right\rangle =\left(\lambda\cdot\left\langle \alpha,x\right\rangle \right)\left\langle \beta,y\right\rangle ,\\
\left\langle \left(\lambda\alpha\right)\beta,xy\right\rangle =\left\langle \left(\alpha\lambda\right)\beta,xy\right\rangle =\left\langle \alpha\left(\lambda\beta\right),xy\right\rangle =\left\langle \alpha,x\right\rangle \left\langle \lambda\beta,y\right\rangle =\left\langle \alpha,x\right\rangle \left(\lambda\cdot\left\langle \beta,y\right\rangle \right)
\end{gather*}
for any $\alpha,\beta,\lambda,\kappa\in R$ and $x,y\in M$, so $R\boxtimes M$
is a scalable monoid with $\cdot$ a scaling action of $R$ on $R\times M$.
\end{proof}
\begin{prop}
\label{p21}Let a ring-monoid $R\boxtimes M$ be a scalable monoid
over $R$. Then $\mathcal{U}=\left\{ \left\langle 1,x\right\rangle \mid x\in M\right\} $
is a coherent system of unit elements of $R\boxtimes M$.
\end{prop}
\begin{proof}
We have $\left\langle \alpha,x\right\rangle =\alpha\cdot\left\langle 1,x\right\rangle $
for any $\left\langle \alpha,x\right\rangle \in R\times M$, and if
$\alpha\cdot\left\langle 1,x\right\rangle =\alpha'\cdot\left\langle 1,x\right\rangle $
then $\left\langle \alpha,x\right\rangle =\left\langle \alpha',x\right\rangle $,
so $\alpha=\alpha'$. Also, $1\cdot\left\langle \alpha,x\right\rangle =\left\langle \alpha,x\right\rangle =\alpha\cdot\left\langle 1,x\right\rangle $,
so $\left\langle \alpha,x\right\rangle \sim\left\langle 1,x\right\rangle $.
Furthermore, if $\left\langle 1,x\right\rangle \sim\left\langle 1,y\right\rangle $
so that $\alpha\cdot\left\langle 1,x\right\rangle =\beta\cdot\left\langle 1,y\right\rangle $
for some $\alpha,\beta\in R$ then $\left\langle \alpha,x\right\rangle =\left\langle \beta,y\right\rangle $,
so $x=y$, so $\left\langle 1,x\right\rangle =\left\langle 1,y\right\rangle $.
Hence, $\mathcal{U}$ is a system of unit elements of $R\boxtimes M$. 

Finally, if $\left\langle 1,x\right\rangle ,\left\langle 1,y\right\rangle \in\mathcal{U}$
then $x,y,xy\in M$, so $\left\langle 1,x\right\rangle \left\langle 1,y\right\rangle =\left\langle 1,xy\right\rangle \in\mathcal{U}$,
and $\left\langle 1,\boldsymbol{1}\right\rangle \in\mathcal{U}$ since
$\boldsymbol{1}\in M$. Hence, $\mathcal{U}$ is a submonoid of $R\boxtimes M$.
\end{proof}
Together, Propositions \ref{p17}, \ref{p21}, \ref{p11} and \ref{p91}
imply the following fact:
\begin{prop}
A ring-monoid $R\boxtimes M$ is an additive scalable monoid if and
only if $R$ is commutative.
\end{prop}
If $R$ is in addition an ordered ring then $R\boxtimes M$ is an
ordered additive scalable monoid by Proposition \ref{p15}. 

Definition \ref{d13} is thus a construction-definition of a scalable
monoid in the case when $R$ is commutative. In Part \ref{sec:Quantity-spaces},
it will become obvious that a ring monoid $R\boxtimes M$, where $R$
is a field and $M$ is a free abelian group, is a quantity space.
This is similar to the construction-definition of a quantity space
given by Carlson \cite{CARL}.\footnote{Recall that Carlson considers ring-monoids of the form $\mathbb{R}\boxtimes G$
where $G$ is an abelian group equipped with an external operation
\begin{gather*}
\mathbb{Q}\times G\rightarrow G,\qquad\left(t,x\right)\mapsto x^{t}
\end{gather*}
such that $G$ is a multiplicatively written vector space over $\mathbb{Q}$,
specifically assumed to be finite-dimensional. This is an unnecessary
assumption, however; it suffices to assume that $G$ is a free module
over $\mathbb{Z},$ or equivalently a free abelian group. In Raposo's
definition of a quantity space \cite{RAP}, Carlson's vector space
of ''pre-units'' is replaced by a finitely generated free abelian
group of dimensions.} 

\part{\label{sec:Quantity-spaces}Quantity spaces}

A quantity space is basically a scalable monoid equipped with property
(A2) in Section \ref{sec:1-4}. Section \ref{s21} below introduces
the formal definition of quantity spaces; some basic facts about quantity
spaces are presented in Section \ref{s23}. 

Systems of unit quantities for quantity spaces are discussed in Section
\ref{s24}. The notion of a measure of a quantity is formally defined
in Section \ref{s25-1}, and ways in which measures serve as proxies
for quantities are described. 

In Section \ref{s26}, we show that the monoid of dimensions $Q/{\sim}$
corresponding to a quantity space $Q$ is a free abelian group and
that bases in $Q$ and $Q/{\sim}$ have the same cardinality. 

A representation of quantity spaces in terms of Laurent monomials
in $n$ indeterminates is presented in Section \ref{sec:3-16}. 

\section{\label{s21}From scalable monoids to quantity spaces}

In this section, we specialize scalable monoids in order to obtain
a mathematical model suitable for calculation with quantities, a quantity
space. The results in Sections \ref{sec:2-9} and \ref{s25} strongly
suggest that a scalable monoid serving this purpose should be equipped
with a sufficiently well-behaved set of unit elements. The simplest
approach is to require that a quantity space is equipped with a coherent
system of unit elements, which is a dense, sparse, closed and consistent
set of unit elements.

A coherent system of unit elements of a scalable monoid corresponds
to what is called a coherent system of units in metrology. There,
coherent systems of units are commonly derived from sets of so-called
base units, such as the three base units in the CGS system. The notion
corresponding to a set of base units here is a basis in a quantity
space, analogous to that of a basis in a vector space or a free abelian
group.

If, for example, $B=\left\{ u_{1},u_{2},u_{3}\right\} $ is a set
of base units, then the set of all quantities of the form $u_{1}^{k_{1}}u_{2}^{k_{2}}u_{3}^{k_{3}}$,
where $k_{i}$ are integers and $u_{i}^{0}=\boldsymbol{1}$, is a
coherent system of units derived from $B$, provided that the product
of $u_{1}^{i_{1}}u_{2}^{i_{2}}u_{3}^{i_{3}}$ and $u_{1}^{j_{1}}u_{2}^{j_{2}}u_{3}^{j_{3}}$
is equal to a term of the form $u_{1}^{k_{1}}u_{2}^{k_{2}}u_{3}^{k_{3}}$.
It is natural to require that
\[
\left(u_{1}^{i_{1}}u_{2}^{i_{2}}u_{3}^{i_{3}}\right)\left(u_{1}^{j_{1}}u_{2}^{j_{2}}u_{3}^{j_{3}}\right)=u_{1}^{\left(i_{1}+j_{1}\right)}u_{2}^{\left(i_{2}+j_{2}\right)}u_{3}^{\left(i_{3}+j_{3}\right)},
\]
but this identity implies that the base units commute. (For example,
$u_{1}u_{2}=u_{1}^{1}u_{2}^{1}u_{3}^{0}=\left(u_{1}^{0}u_{2}^{1}u_{3}^{0}\right)\left(u_{1}^{1}u_{2}^{0}u_{3}^{0}\right)=u_{2}u_{1}$.)
Thus, if we want to include in our model the feature that a coherent
system of units can be derived from a system of base units, there
are good reasons to require the specialized scalable monoids to be
commutative. 

As we have seen, if $X$ is a scalable monoid over $R$ that has a
unit element then $R$ is commutative. Furthermore, if we want to
deal with derived units such as $u_{1}^{1}u_{2}^{-1}u_{3}^{0}$ (meter
per second, etc.) then inverses of units must be admitted, and there
is a close connection between inverses of quantities in a scalable
monoid over $R$ and multiplicative inverses in $R$. This suggests,
finally, that quantity spaces should be defined as certain scalable
monoids over fields.

We now come to the basic definitions motivated by the considerations
above.
\begin{defn}
\label{d2.3}Let $Q$ be a commutative scalable monoid over $R$.
A \emph{finite} \emph{set of generators} for $Q$ is a set $B=\left\{ b_{1},\ldots,b_{n}\right\} $
of elements of $Q$ such that every $x\in Q$ has an expansion 
\[
x=\mu\cdot\prod_{i=1}^{n}b_{i}^{_{_{k_{i}}}},
\]
where $\mu\in R$ and $k_{1},\dots,k_{n}$ are integers. A \emph{finite
basis} for $Q$ is a finite set of generators for $Q$ such that every
$x\in Q$ has a \emph{unique} expansion of this form, and a \emph{strong}
finite basis for $Q$ is a finite basis such that every basis element
is invertible.
\end{defn}
Note that the uniqueness of the expansion means that the array $\left(\mu,k_{1},\ldots,k_{n}\right)$
is unique given an indexing $B\rightarrow\left\{ 1,\ldots,n\right\} $
of the basis elements. We say that a commutative scalable monoid with
a (strong) finite basis is \emph{(strongly) free}. 
\begin{defn}
\label{d2.4}A finitely generated\emph{ quantity space} is a commutative
scalable monoid $Q$ over a field, such that there exists a strong\emph{
}finite\emph{ }basis\emph{ }for\emph{ $Q$.}
\end{defn}
Elements of a quantity space are called \emph{quantities,} unit elements
are called \emph{unit quantities}, and orbit classes in a quantity
space are called \emph{dimensions}.

It is not very complicated to generalize the notion of a finite basis
for a commutative scalable monoid to include possibly infinite bases,
and thus to generalize finitely generated quantity spaces accordingly,
but in view of the connection to metrology only the special case of
finitely generated quantity spaces will be considered below.

\section{\label{s23}Some basic facts about quantity spaces}
\begin{prop}
\label{p19a}Let $Q$ be a quantity space with a strong basis $\left\{ b_{1},\ldots,b_{n}\right\} $
and $x,y\in Q$. We have
\begin{enumerate}
\item $\boldsymbol{1}=1\cdot\prod_{i=1}^{n}b_{i}^{0}$\emph{;}
\item if $x=\mu\cdot\prod_{i=1}^{n}b_{i}^{_{_{k_{i}}}}$ and $y=\nu\cdot\prod_{i=1}^{n}b_{i}^{\ell_{i}}$
then
\[
xy=\mu\nu\cdot\prod_{i=1}^{n}b_{i}^{_{\left(k_{i}+\ell_{i}\right)}};
\]
\item if $x=\mu\cdot\prod_{i=1}^{n}b_{i}^{_{_{k_{i}}}}$ and $\mu\neq0$
then $x^{-1}=\frac{1}{\mu}\cdot\prod_{i=1}^{n}b_{i}^{_{_{-k_{i}}}}$.
\end{enumerate}
\end{prop}
\begin{proof}
To prove (1), note that $b_{i}^{0}=\boldsymbol{1}$ for all $b_{i}$.
(2) follows from Lemma \ref{thm:lem1} and the fact that $Q$ is commutative.
(3) follows from (1) and (2).
\end{proof}
\begin{prop}
\label{p19}If $Q$ be a (finitely generated) quantity space then
every element of a basis $B$ for $Q$ is non-zero.
\end{prop}
\begin{proof}
We have $0_{\mathsf{C}}=0\cdot x$ and $x=1\cdot x$ for all $x\in\mathsf{C}$,
so if $0_{\mathsf{C}}\in B$ then $0_{\mathsf{C}}=0\cdot0_{\mathsf{C}}$
and $0_{\mathsf{C}}=1\cdot0_{\mathsf{C}}$ are expansions of $0_{\mathsf{C}}$
in terms of $B$, so $0_{\mathsf{C}}$ does not have a unique expansion
in terms of $B$ since $0\neq1$ in a field, so $B$ is not a basis
for $Q$.
\end{proof}
\begin{prop}
\label{p20}If $Q$ be a quantity space with a basis $\left\{ b_{1},\ldots b_{n}\right\} $
then $x\in Q$ is a non-zero quantity if and only if $\mu\neq0$ in
the expansion $x=\mu\cdot\prod_{i=1}^{n}b_{i}^{_{_{k_{i}}}}$.
\end{prop}
\begin{proof}
We have $0\cdot x=0\cdot\left(\mu\cdot\prod_{i=1}^{n}b_{i}^{_{_{k_{i}}}}\right)=0\cdot\prod_{i=1}^{n}b_{i}^{_{_{k_{i}}}}$,
so if $\mu=0$ then $0\cdot x=x$ and if $0\cdot x=x=\mu\cdot\prod_{i=1}^{n}b_{i}^{_{_{k_{i}}}}$
then $\mu=0$ since the expansion of $x$ is unique.
\end{proof}
In particular, $\mathbf{1}$ is a non-zero quantity.
\begin{prop}
\label{p23}If $Q$ is a (finitely generated) quantity space then
the product of non-zero quantities in $Q$ is a non-zero quantity.
\end{prop}
\begin{proof}
Set $x=\mu\cdot\prod_{i=1}^{n}b_{i}^{_{_{k_{i}}}}$ and $y=\nu\cdot\prod_{i=1}^{n}b_{i}^{_{\ell_{i}}}$.
Then $xy=\mu\nu\cdot\prod_{i=1}^{n}b_{i}^{\left(k_{i}+\ell_{i}\right)}$,
and $\mu\nu\neq0$ since there are no zero divisors in a field.
\end{proof}
\begin{prop}
\label{p22-1}If $Q$ be a (finitely generated) quantity space then
$x\in Q$ is invertible if and only if $x$ is non-zero.
\end{prop}
\begin{proof}
Let $\left\{ b_{1},\ldots b_{n}\right\} $ be a strong basis for $Q$
so that $x=\mu\cdot\prod_{i=1}^{n}b_{i}^{_{_{k_{i}}}}$. If $\mu\neq0$
then $\frac{1}{\mu}\cdot\prod_{i=1}^{n}b_{i}^{_{_{-k_{i}}}}$ is an
inverse of $x$; conversely, if $\mu=0$ then there is no $\nu$ such
that $\mu\nu=1$ in a field, so $x$ does not have an inverse $\nu\cdot\prod_{i=1}^{n}b_{i}^{\ell_{i}}$.
\end{proof}
Combining Propositions \ref{p19} and \ref{p22-1}, we obtain the
following result. 
\begin{prop}
Every element of a basis for a (finitely generated) quantity space
is invertible.
\end{prop}
In other words, if there exists a strong finite basis for a commutative
scalable monoid $Q$ over a field then every finite basis for $Q$
is a strong basis. Accordingly, the distinction between a basis and
a strong basis for a quantity space is redundant; it suffices to talk
about a basis. We also note that a strongly free quantity space is
just a free quantity space.
\begin{lem}
\label{s3.3}Let $Q$ be a quantity space over $K$ with a basis $\left\{ b_{1},\ldots b_{n}\right\} $,
and consider $x=\mu\cdot\prod_{i=1}^{n}b_{i}^{_{_{k_{i}}}}$ and $y=\nu\cdot\prod_{i=1}^{n}b_{i}^{_{_{\ell_{i}}}}$.
The following conditions are equivalent: 
\begin{enumerate}
\item $x\sim y$, or equivalently $\mu\cdot\prod_{i=1}^{n}b_{i}^{_{_{k_{i}}}}\sim\nu\cdot\prod_{i=1}^{n}b_{i}^{_{_{\ell_{i}}}}$\emph{;} 
\item $k^{i}=\ell{}^{i}$ for $i=1,\ldots,n$\emph{;} 
\item $\prod_{i=1}^{n}b_{i}^{_{_{k_{i}}}}=\prod_{i=1}^{n}b_{i}^{_{_{\ell_{i}}}}$\emph{;}
\item $\nu\cdot x=\mu\cdot y$, or equivalently $\nu\cdot\left(\mu\cdot\prod_{i=1}^{n}b_{i}^{_{_{k_{i}}}}\right)=\mu\cdot\left(\nu\cdot\prod_{i=1}^{n}b_{i}^{_{_{\ell_{i}}}}\right)$. 
\end{enumerate}
\end{lem}
\begin{proof}
The implications $(2)\Longrightarrow(3)$ , $(3)\Longrightarrow(4)$
and $(4)\Longrightarrow(1)$ are trivial. To prove $(1)\Longrightarrow(2)$,
note that if $x\sim y$ so that $\alpha\cdot\left(\mu\cdot\prod_{i=1}^{n}b_{i}^{_{_{k_{i}}}}\right)=\beta\cdot\left(\nu\cdot\prod_{i=1}^{n}b_{i}^{_{_{\ell_{i}}}}\right)$
for some $\alpha,\beta\in K$ then 
\[
\alpha\mu\cdot\prod_{i=1}^{n}b_{i}^{_{_{k_{i}}}}=z=\beta\nu\cdot\prod_{i=1}^{n}b_{i}^{_{_{\ell_{i}}}}.
\]
As the expansion of $z$ is unique, $k^{i}=\ell{}^{i}$ for $i=1,\ldots,n$.
\end{proof}
Note that Lemma \ref{s3.3} implies Fourier's principle of dimensional
homogeneity \cite{FOUR}: if not $k^{i}=\ell{}^{i}$ for $i=1,\ldots,n$
then $x\nsim y$, so $x\neq y$.

\section{\label{s24}Systems of unit quantities in quantity spaces}
\begin{prop}
\label{p22}If $Q$ is a (finitely generated) quantity space then
every non-zero quantity $u\in Q$ is a unit quantity for $\left[u\right]$.
\end{prop}
\begin{proof}
Set $u=\mu\cdot\prod_{i=1}^{n}b_{i}^{_{_{k_{i}}}}$ and $x=\nu\cdot\prod_{i=1}^{n}b_{i}^{_{_{\ell_{i}}}}$.
Then $\mu\neq0$ by Proposition \ref{p20}, and if $u\sim x$ then
$\nu\cdot u=\mu\cdot x$ by Lemma \ref{s3.3}, so $x=\mu^{-1}\mu\cdot x=\mu^{-1}\cdot\left(\mu\cdot x\right)=\mu^{-1}\cdot\left(\nu\cdot u\right)=\mu^{-1}\nu\cdot u$.

Also, if $\lambda\cdot u=\lambda'\cdot u$ then $\lambda\mu\cdot\prod_{i=1}^{n}b_{i}^{k_{i}}=z=\lambda'\mu\cdot\prod_{i=1}^{n}b_{i}^{k_{i}}$,
so $\lambda\mu=\lambda'\mu$ since the expansion of $z$ is unique,
so $\lambda=\lambda'$ since $\mu\neq0$. 
\end{proof}
\begin{cor}
If $Q$ is a (finitely generated) quantity space then a dense, sparse
set of non-zero elements of $Q$ is a system of unit quantities of
$Q$, and a sparse set of non-zero elements of $Q$ is a partial system
of unit quantities of $Q$.
\end{cor}
In metrology, unit quantities are called \emph{measurement units \cite{JCGM}}.
A \emph{set of base units} $B$ is a finite set of measurement units
each of which cannot be expressed as a product of powers of the other
measurement units in $B$ \cite{JCGM}. A finite basis $B$ for a
quantity space $Q$ is a set of base units; if $b\in B$ is not a
base unit relative to $B$ then $1\cdot b=b=\prod_{i=1}^{n}b_{i}^{_{_{k_{i}}}}=1\cdot\prod_{i=1}^{n}b_{i}^{_{_{k_{i}}}}$,
where $b_{i}\in B$ and $b_{i}\neq b$ for all $b_{i}$, so $b$ does
not have a unique expansion relative to $B$, so $B$ is not a basis
for $Q$. 
\begin{prop}
\label{p26-1}If $Q$ is a quantity space with basis $B=\left\{ b_{1},\ldots,b_{n}\right\} $
then 
\[
\mathcal{U}=\left\{ 1\cdot\prod_{i=1}^{n}b_{i}^{k_{i}}\mid k_{i}\in\mathbb{Z}\right\} 
\]
 is a coherent system of unit quantities for $Q$.
\end{prop}
\begin{proof}
All elements of $B$ are non-zero by Proposition \ref{p19}, so all
elements of $\mathcal{U}$ are non-zero and hence unit quantities
by Proposition \ref{p22}. Also, $\mathcal{U}$ is dense in $Q$ since
every $x\in Q$ has an expansion $x=\mu\cdot\prod_{i=1}^{n}b_{i}^{_{_{k_{i}}}}$,
so $1\cdot x=\mu\cdot\left(1\cdot\prod_{i=1}^{n}b_{i}^{k_{i}}\right)$.
Lastly, if $u=1\cdot\prod_{i=1}^{n}b_{i}^{k_{i}}\sim1\cdot\prod_{i=1}^{n}b_{i}^{_{_{\ell_{i}}}}=v$
then $\prod_{i=1}^{n}b_{i}^{k_{i}}=\prod_{i=1}^{n}b_{i}^{_{_{\ell_{i}}}}$
by Lemma \ref{s3.3}, so $u=v$, meaning that $\mathcal{U}$ is sparse
in $Q$.

It remains to prove that $\mathcal{U}$ is a monoid. Clearly, $\mathbf{1}\in\mathcal{U}$
since $\boldsymbol{1}=1\cdot\prod_{i=1}^{n}b_{i}^{0}$, and we have
\[
\left(1\cdot\prod_{i=1}^{n}b_{i}^{_{_{k_{i}}}}\right)\left(1\cdot\prod_{i=1}^{n}b_{i}^{\ell_{i}}\right)=1\cdot\prod_{i=1}^{n}b_{i}^{\left(k_{i}+\ell_{i}\right)},
\]
so if $u,v\in\mathcal{U}$ then $uv\in\mathcal{U}$. Thus, $\mathcal{U}$
is a submonoid of $Q$.
\end{proof}
In other words, every (finite) basis $B$ can be extended to a coherent
system $\mathcal{U}$ of unit quantities, consisting of basis quantities
and other unit quantities that are expressed as products of basis
quantities and their inverses.

In metrology, a \emph{coherent system of units} $U$ is defined essentially
as a set of measurement units each of which is either a base unit
$b_{I}\in U$ or a \emph{coherent derived unit, }a non-base unit of
the form $1\cdot\prod_{i=1}^{n}b_{i}^{k_{i}}$, where each $b_{i}$
is a base unit in $U$ and $k_{1},\ldots,k_{n}$ are integers \cite{JCGM}.
By Proposition \ref{p26-1}, a coherent system of units in this sense
is a coherent system of unit quantities in the sense of Definition
\ref{d12}.

By Propositions \ref{p26-1} and \ref{p11}, every (finitely generated)
quantity space is an \linebreak{}
additive quantity space in the sense of Definition \ref{d9}.

Also, Propositions \ref{p26-1} and \ref{p15} imply that if $Q$
is a (finitely generated) quantity space over an ordered field $K$
then $Q$ is an ordered additive quantity space over $K$ with $\mathsf{C}$
ordered by $\leq_{\mathsf{C}}$ defined by $x\leq_{\mathsf{C}}y$
if and only if $y-x=\rho\cdot u$ for some $\rho\in K$ such that
$0\leq\rho$ and some $u\in\mathcal{U}\cap\mathsf{C}$, where $\mathcal{U}$
is a coherent system of unit quantities derived from a basis for $Q$.
Thus, every quantity space over $\mathbb{Q}$ or $\mathbb{R}$ can
be regarded as an ordered additive quantity space since $\mathbb{Q}$
and $\mathbb{R}$ are ordered. We normally want quantity spaces to
be ordered, and since any Dedekind-complete ordered field is isomorphic
to $\mathbb{R}$ \cite{BLY}, it is natural to let $K$ be the real
numbers $\mathbb{R}$.

\section{\label{s25-1}Measures of quantities}
\begin{defn}
\label{d2.5}Let $Q$ be a quantity space over $K$ with basis $B=\left\{ b_{1},\ldots,b_{n}\right\} $.
The uniquely determined scalar $\mu\in K$ in the expansion 
\[
x=\mu\cdot\prod_{i=1}^{n}b_{i}^{k_{i}}
\]
is called the \emph{measure} of $x$ relative to \textbf{$B$} and
will be denoted by $\mu_{B}\left(x\right)$. $\hphantom{\square}\square$ 
\end{defn}
For example, $\mathbf{1}=1\cdot\prod_{i=1}^{n}b_{i}^{0}$ for any
$B$, so we have the following simple but useful fact.
\begin{prop}
\label{p25}Let $Q$ be a (finitely generated) quantity space. For
any basis $B$ for $Q$ we have $\mu_{B}\left(\mathbf{1}\right)=1$.
\end{prop}
Relative to a fixed basis, measures of quantities can be used as proxies
for the quantities themselves. 
\begin{prop}
\label{p26}If $Q$ is a (finitely generated) quantity space with
basis $B$ and $x,y\in Q$ then $\mu_{B}\left(xy\right)=\mu_{B}\left(x\right)\mu_{B}\left(y\right)$.
\end{prop}
\begin{proof}
This follows immediately from Proposition \ref{p19a}.
\end{proof}
\begin{prop}
\label{s2.6}Let $Q$ be a (finitely generated) quantity space with
basis $B$. A quantity $x\in Q$ is invertible if and only if $\mu_{B}\left(x\right)\neq0$,
and for any invertible $x\in Q$ we have $\mu_{B}\left(x^{-1}\right)=\mu_{B}\left(x\right)^{-1}$.
\end{prop}
\begin{proof}
The first part of the assertion follows from Propositions \ref{p20}
and \ref{p22-1}; the second part follows Propositions \ref{p19a},
\ref{p20} and \ref{p22-1}.
\end{proof}
\begin{prop}
\label{p28}If $Q$ is a (finitely generated) quantity space with
basis $B$ then $\mu_{B}\left(\lambda\cdot x\right)=\lambda\,\mu_{B}\left(x\right)$
for all $\lambda\in K$ and $x\in Q$.
\end{prop}
\begin{proof}
If $x=\mu\cdot\prod_{i=1}^{n}b_{i}^{_{_{k_{i}}}}$ then $\lambda\cdot x=\lambda\cdot\left(\mu\cdot\prod_{i=1}^{n}b_{i}^{_{_{k_{i}}}}\right)=\lambda\mu\cdot\prod_{i=1}^{n}b_{i}^{_{_{k_{i}}}}$,
so $\mu_{B}\left(\lambda\cdot x\right)=\lambda\mu=\lambda\,\mu_{B}\left(x\right)$.
\end{proof}
\begin{prop}
\label{s3.5}If $Q$ is a (finitely generated) quantity space with
basis $B$ then $\mu_{B}\left(x\right)+\mu_{B}\left(y\right)=\mu_{B}\left(x+y\right)$
for all $x,y\in X$ such that $x\sim y$.
\end{prop}
\begin{proof}
Let $x=\mu_{B}\left(x\right)\cdot\prod_{i=1}^{n}b_{i}^{k_{i}}$ and
$y=\mu_{B}\left(y\right)\cdot\prod_{i=1}^{n}b_{i}^{k_{i}}$ be the
expansions of $x$ and $y$ relative to $B=\left\{ b_{1},\ldots.b_{n}\right\} $.
As $\prod_{i=1}^{n}b_{i}^{k_{i}}$ is non-zero, and thus a unit quantity
for $\left[\prod_{i=1}^{n}b_{i}^{k_{i}}\right]$ by Proposition \ref{p22},
we have 
\[
x+y=\mu_{B}\left(x\right)\cdot\prod_{i=1}^{n}b_{i}^{k_{i}}+\mu_{B}\left(y\right)\cdot\prod_{i=1}^{n}b_{i}^{k_{i}}=\left(\mu_{B}\left(x\right)+\mu_{B}\left(y\right)\right)\cdot\prod_{i=1}^{n}b_{i}^{k_{i}},
\]
proving the assertion. 
\end{proof}
In general, the measure of a quantity depends on a choice of basis,
but there is an important exception to this rule.
\begin{prop}
\label{s3.4}Let $Q$ be a (finitely generated) quantity space over
$K$. For every $x\in\left[\mathbf{1}\right]$, the measure $\mu_{B}\left(x\right)$
of $x$ relative to a basis $B$ for $Q$ does not depend on $B$.
\end{prop}
\begin{proof}
$\mathbf{1}$ is a unit quantity for $\left[\mathbf{1}\right]$ by
Proposition \ref{p22}, so there is a unique $\lambda\in K$ such
that $x=\lambda\cdot\mathbf{1}$, so $\mu_{B}\left(x\right)=\lambda\,\mu_{B}\left(\mathbf{1}\right)$
by Proposition \ref{p28}, and $\mu_{B}\left(\mathbf{1}\right)$ does
not depend on $B$ by Proposition \ref{p25}. 
\end{proof}
It is common to refer to a quantity $x\in\left[\mathbf{1}\right]$
as a ``dimensionless quantity'', although $x$ is not really ``dimensionless''
\textendash{} it belongs to, or ``has'', the dimension $\left[\mathbf{1}\right]$.
The so-called Buckingham $\pi$ theorem in dimensional analysis depends
on the fact stated in Proposition \ref{s3.4}.

\section{\label{s26}Groups of dimensions; cardinality of bases}

Recall that a trivially scalable monoid $Q/{\sim}$ may also be regarded
as a plain monoid. The definition of a basis for a commutative monoid
differs slightly from that for a scalable commutative monoid.
\begin{defn}
\label{d17}Let $M$ be a commutative monoid. A \emph{finite} \emph{basis}
for $M$ is a set $B=\left\{ b_{1},\ldots,b_{n}\right\} $ of elements
of $M$ such that every $x\in M$ has a unique expansion 
\[
x=\prod_{i=1}^{n}b_{i}^{_{_{k_{i}}}},
\]
where $k_{1},\dots,k_{n}$ are integers. A \emph{strong} finite basis
for $M$ is a finite basis for $M$ such that every basis element
is invertible.
\end{defn}
In this section, every quotient of the form $Q/{\sim}$, where $Q$
is a quantity space, will be regarded as a monoid, which means that
Definition \ref{d17} will be used instead of Definition \ref{d2.3}
in these cases.
\begin{prop}
If $Q$ is a (finitely generated) quantity space then $Q/{\sim}$
is an abelian group.
\end{prop}
\begin{proof}
$Q/{\sim}$ is commutative since $\left[x\right]\left[y\right]=\left[xy\right]=\left[yx\right]=\left[y\right]\left[x\right]$
for all $\left[x\right],\left[y\right]\in Q/{\sim}$. To prove that
$Q/{\sim}$ is a group it suffices to show that for every $\left[x\right]\in X/{\sim}$
there is a dimension $\left[x\right]^{-1}\in X/{\sim}$ such that
$\left[x\right]\left[x\right]^{-1}=\left[x\right]^{-1}\left[x\right]=\left[\mathbf{1}\right]$.
Let $B=\left\{ b_{1},\ldots,b_{n}\right\} $ be a basis for $Q$ and
let $x=\mu\cdot\prod_{i=1}^{n}b_{i}^{k_{i}}$ be the unique expansion
of $x$ relative to $B$. Also set $y=1\cdot\prod_{i=1}^{n}b_{i}^{-k_{i}}$.
Then $\left[x\right]\left[y\right]=\left[xy\right]=\left[\mu\cdot\boldsymbol{1}\right]=\left[yx\right]=\left[y\right]\left[x\right]$,
so $\left[y\right]$ is an inverse $\left[x\right]^{-1}$ of $\left[x\right]$
since $\left[\mu\cdot\boldsymbol{1}\right]=\left[\boldsymbol{1}\right]$. 
\end{proof}
By a finite basis for an abelian group $G$ we mean a (strong) finite
basis for $G$ as a commutative monoid, and a finitely generated free
abelian group is an abelian group for which such a finite basis exists.
\begin{prop}
\label{s3.7-1}If $Q$ is a quantity space over $K$ with basis $B=\left\{ b_{1},\ldots,b_{n}\right\} $,
then $\mathsf{B}=\left\{ \left[b_{1}\right],\ldots,\left[b_{n}\right]\right\} $
is a basis for $Q/{\sim}$ with the same cardinality as $B$. 
\end{prop}
\begin{proof}
The unique expansions of $b_{i},b_{i'}\in B$ relative to $B$ are
$b_{i}=1\cdot b_{i}$ and $b_{i'}=1\cdot b_{i'}$. Hence, $\left[b_{i}\right]=\left[b_{i'}\right]$
implies $b_{i}=b_{i'}$ since $b_{i}\sim b_{i'}$ implies $1\cdot b_{i}=1\cdot b_{i'}$
by Lemma \ref{s3.3}, so the surjective mapping $\phi:B\rightarrow\mathsf{B}$
given by $\phi\left(b_{i}\right)=\left[b_{i}\right]$ is injective
as well and hence a bijection.

Now, let $\left[x\right]$ be an arbitrary dimension in $Q/{\sim}$.
As $B$ is a basis for $Q$, we have $x=\mu\cdot\prod_{i=1}^{n}b_{i}^{k_{i}}$
for some $\mu\in K$ and some integers $k_{1},\ldots,k_{n}$, so $\left[x\right]=\left[\mu\cdot\prod_{i=1}^{n}b_{i}^{k_{i}}\right]=\left[\prod_{i=1}^{n}b_{i}^{k_{i}}\right]=\prod_{i=1}^{n}\left[b_{i}\right]^{k_{i}}$.
Also, if $\left[x\right]=\prod_{i=1}^{n}\left[b_{i}\right]^{k_{i}}=\prod_{i=1}^{n}\left[b_{i}\right]^{\ell{}_{i}}$,
then $\left[\prod_{i=1}^{n}b_{i}^{k_{i}}\right]=\left[\prod_{i=1}^{n}b_{i}^{\ell{}_{i}}\right]$,
so $1\cdot\prod_{i=1}^{n}b_{i}^{k_{i}}\sim1\cdot\prod_{i=1}^{n}b_{i}^{\ell{}_{i}}$
so $k_{i}=\ell_{i}$ for $i=1,\ldots,n$ by Lemma \ref{s3.3}. Hence,
$\mathsf{B}$ is a basis for $Q/{\sim}$.
\end{proof}
We say that a basis $\left\{ b_{1},\ldots,b_{n}\right\} $ for $Q$
and a basis $\left\{ \mathsf{b}_{1},\ldots,\mathsf{b}_{m}\right\} $
for $Q/{\sim}$ are \emph{similar} when $m=n$ and $\left[b_{i}\right]=\mathsf{b}_{i}$
for $i=1,\ldots,n$.
\begin{cor}
\label{c6}Let $Q$ be a (finitely generated) quantity space. For
every basis for $Q$ there exists a unique similar basis for $Q/{\sim}$. 
\end{cor}
Hence, corresponding to the fact that if $X$ is a scalable monoid
then $X/{\sim}$ is a monoid, we have the following much stronger
result.
\begin{prop}
If $Q$ is a (finitely generated) quantity space then $Q/{\sim}$
is a (finitely generated) free abelian group. 
\end{prop}
The idea that the set of dimensions of a quantity space forms a free
abelian group is present in articles by Krystek \cite{KRYS} and Raposo
\cite{RAP}. This is actually an assumption built into the definition
of quantity spaces in \cite{RAP}; here it is a fact derived from
the definitions of quantity spaces and commensurability relations
on quantity spaces. 

A finitely generated abelian group may have no finite basis; in this
case, a \linebreak{}
corresponding finitely generated trivially scalable commutative monoid
over a field cannot have a finite basis since this would contradict
Proposition \ref{s3.7-1}. Thus, a finitely generated commutative
scalable monoid over a field need not be a finitely generated quantity
space. (This may be generalized to the case of infinite bases.)
\begin{prop}
\label{p33}If $Q$ is a (finitely generated) quantity space and $\mathsf{B}=\left\{ \mathsf{b}_{1},\ldots,\mathsf{b}_{n}\right\} $
is a basis for $Q/{\sim}$ such that for each $\mathsf{b}_{i}\in\mathsf{B}$
there is is a non-zero quantity $b_{i}\in Q$ such that $\mathsf{b}_{i}=\left[b_{i}\right]$
then $B=\left\{ b_{1},\ldots,b_{n}\right\} $ is a basis for $Q$
with the same cardinality as $\mathsf{B}$.
\end{prop}
\begin{proof}
Consider the function $\psi:\mathsf{B}\rightarrow\psi\left(\mathsf{B}\right)$
given by $\psi\left(\mathsf{b}_{i}\right)=b_{i}$. $\psi$ is surjective,
and we have $\psi\left(\mathsf{B}\right)=\left\{ b_{1},\ldots,b_{n}\right\} $.
Also, if $\left[b_{i}\right]\neq\left[b_{i'}\right]$ then $b_{i}\neq b_{i'}$
since dimensions are disjoint, meaning that  $\psi$ is injective
as well and hence a bijection.

Let $x$ be an arbitrary quantity in $Q$. As $\mathsf{B}$ is a basis
for $Q/{\sim}$, we have $\left[x\right]=\prod_{i=1}^{n}\left[b_{i}\right]^{k_{i}}=\left[\prod_{i=1}^{n}b_{i}^{k_{i}}\right]$
for some integers $k_{1},\ldots,k_{n}$, and if $b_{i}\neq0$ for
each $b_{i}$ then $\prod_{i=1}^{n}b_{i}^{k_{i}}$ is non-zero and
thus a unit quantity for $\left[x\right]$ by Proposition \ref{p22},
so there exists a unique $\mu\in K$ such that $x=\mu\cdot\prod_{i=1}^{n}b_{i}^{k_{i}}$. 

Also, if $x=\mu\cdot\prod_{i=1}^{n}b_{i}^{k_{i}}=\nu\cdot\prod_{i=1}^{n}b_{i}^{\ell_{i}}$
then $\left[\mu\cdot\prod_{i=1}^{n}b_{i}^{k_{i}}\right]=\left[\nu\cdot\prod_{i=1}^{n}b_{i}^{\ell_{i}}\right]$,
so $\left[\prod_{i=1}^{n}b_{i}^{k_{i}}\right]=\left[\prod_{i=1}^{n}b_{i}^{\ell_{i}}\right]$,
so $\prod_{i=1}^{n}\left[b_{i}\right]^{k_{i}}=\prod_{i=1}^{n}\left[b_{i}\right]^{\ell_{i}}$,
so $k_{i}=\ell_{i}$ for $i=1,\ldots,n$, since $\mathsf{B}$ is a
basis for $Q/\sim$, so $\nu=\mu$ by the uniqueness of $\mu$. We
have thus shown that $B$ is a basis for $Q$.
\end{proof}
As $\mathsf{b}_{i}=\left[\psi\left(\mathsf{b}_{i}\right)\right]$
and $\psi$ is a bijection between $\mathsf{B}$ and $B$, Proposition
\ref{p33} implies the following fact.
\begin{cor}
\label{c8}Let $Q$ be a (finitely generated) quantity space. For
every basis for $Q/{\sim}$ there exists a similar basis for $Q$.
\end{cor}
The next fact concerns the invariance of basis numbers.
\begin{prop}
\label{s3.8}Let $Q$ be a (finitely generated) quantity space. Any
two bases for $Q$ and any two bases for $Q/{\sim}$ have the same
cardinality. 
\end{prop}
\begin{proof}
We use the fact that any two bases of a (finitely generated) free
abelian group have the same cardinality. Thus, any two bases for $Q/{\sim}$
have the same cardinality. Also, if $\left\{ b_{1},\ldots,b_{n}\right\} $
and $\left\{ b_{1}',\ldots,b_{m}'\right\} $ are bases for $Q$ then
$\left\{ \left[b_{1}\right],\ldots,\left[b_{n}\right]\right\} $ and
$\left\{ \left[b_{1}'\right],\ldots,\left[b_{m}'\right]\right\} $
are bases for $Q/{\sim}$ of cardinality $n$ and $m$, respectively,
so $n=m$.
\end{proof}
Recall that a free module of rank $n$ is a module with a basis and
such that all bases have the same cardinality $n$. Defining the rank
of a commutative scalable monoid analogously, we can say that finitely
generated quantity spaces are free of finite rank, as are finite-dimensional
vector spaces.

\pagebreak{}

\section{\label{sec:3-16}Laurent monomials in \emph{n} indeterminates as
quantities}

Recall that if $K$ is a field then $K^{n}$ is a vector space with
operations defined by 
\begin{gather*}
\lambda\cdot\left(x_{1},\ldots,x_{n}\right)=\left(\lambda x_{1},\ldots,\lambda x_{n}\right),\\
\left(x_{1},\ldots,x_{n}\right)+\left(y_{1},\ldots,y_{n}\right)=\left(x_{1}+y_{1},\ldots,x_{n}+y_{n}\right),
\end{gather*}
 and every $n$-dimensional vector space over $K$ is isomorphic to
a vector space $K^{n}$. A quantity space has a similar representation.

We call a term of the form
\begin{equation}
\lambda X_{1}^{k_{1}}\ldots X_{n}^{k_{n}},\label{eq:131}
\end{equation}
where $X_{1},\ldots,X_{n}$ are uninterpreted symbols, $\lambda$
belongs to a field $K$ and $k_{1},\ldots,k_{n}$ are integers, a
\emph{Laurent monomial in $n$ indeterminates}, and we denote the
set of all terms of the form (\ref{eq:131}) by $K\left\llbracket X_{1},\ldots,X_{n}\right\rrbracket $.
\begin{prop}
Let $s=\kappa X_{1}^{j_{1}}\ldots X_{n}^{j_{n}}$ and $t=\lambda X_{1}^{k_{1}}\ldots X_{n}^{k_{n}}$
be any elements of \textup{$K\left\llbracket X_{1},\ldots,X_{n}\right\rrbracket $.}
For any field $K$, $K\left\llbracket X_{1},\ldots,X_{n}\right\rrbracket $
is a finitely generated quantity space with operations $\left(\alpha,t\right)\mapsto\alpha\cdot t$
and $\left(s,t\right)\mapsto st$ defined by 
\begin{gather*}
\alpha\cdot\lambda X_{1}^{k_{1}}\ldots X_{n}^{k_{n}}=\alpha\lambda X_{1}^{k_{1}}\ldots X_{n}^{k_{n}},\\
\kappa X_{1}^{j_{1}}\ldots X_{n}^{j_{n}}\lambda X_{1}^{k_{1}}\ldots X_{n}^{k_{n}}=\kappa\lambda X_{1}^{\left(j_{1}+k_{1}\right)}\ldots X_{n}^{\left(j_{n}+k_{n}\right)}.
\end{gather*}
\end{prop}
\begin{proof}
Multiplication of terms in $K\left\llbracket X_{1},\ldots,X_{n}\right\rrbracket $
is associative and commutative because addition of integers and multiplication
of elements of $K$ is associative and commutative, and setting
\[
\boldsymbol{1}=1X_{1}^{0}\ldots X_{n}^{0}
\]
 we clearly have $\boldsymbol{1}t=t\boldsymbol{1}=t$. Furthermore,
\begin{gather*}
1\cdot t=1\cdot\lambda X_{1}^{k_{1}}\ldots X_{n}^{k_{n}}=1\lambda X_{1}^{k_{1}}\ldots X_{n}^{k_{n}}=t,\\
\alpha\cdot\left(\beta\cdot t\right)=\alpha\cdot\left(\beta\cdot\lambda X_{1}^{k_{1}}\ldots X_{n}^{k_{n}}\right)=\alpha\cdot\beta\lambda X_{1}^{k_{1}}\ldots X_{n}^{k_{n}}=\\
\alpha\beta\lambda X_{1}^{k_{1}}\ldots X_{n}^{k_{n}}=\alpha\beta\cdot\lambda X_{1}^{k_{1}}\ldots X_{n}^{k_{n}}=\alpha\beta\cdot t
\end{gather*}
and
\begin{gather*}
\alpha\kappa X_{1}^{j_{1}}\ldots X_{n}^{j_{n}}\lambda x_{1}^{k_{1}}\ldots x_{n}^{k_{n}}=\alpha\kappa\lambda X_{1}^{\left(j_{1}+k_{1}\right)}\ldots X_{n}^{\left(j_{n}+k_{n}\right)}=\\
\kappa\alpha\lambda X_{1}^{\left(j_{1}+k_{1}\right)}\ldots X_{n}^{\left(j_{n}+k_{n}\right)}=\kappa X_{1}^{j_{1}}\ldots X_{n}^{j_{n}}\alpha\lambda X_{1}^{k_{1}}\ldots X_{n}^{k_{n}}
\end{gather*}
since multiplication in $K$ is associative and commutative, so $\alpha\cdot st=\left(\alpha\cdot s\right)t=s\left(\alpha\cdot t\right)$
since $\alpha\kappa X_{1}^{j_{1}}\ldots X_{n}^{j_{n}}=\alpha\cdot s$
and so on.

$K\left\llbracket X_{1},\ldots,X_{n}\right\rrbracket $ is thus a
commutative scalable monoid over a field; it remains to exhibit a
strong finite basis. For $i=1,\ldots,n$, set $b_{i}=1X_{1}^{\delta_{1}}\cdots X_{n}^{\delta_{n}}$,
where $\delta_{j}=1$ for $i=j$ and $\delta_{j}=0$ for $i\neq j$,
and $b_{i}^{-1}=1X_{1}^{\vartheta_{1}}\cdots X_{n}^{\vartheta_{n}}$,
where $\vartheta_{j}=-1$ for $i=j$ and $\vartheta_{j}=0$ for $i\neq j$,
so that $b_{i}b_{i}^{-1}=b_{i}^{-1}b_{i}=\boldsymbol{1}$. It is clear
that $t$ has a unique representation of the form
\[
t=\lambda\cdot\prod_{i=1}^{n}b_{i}^{k_{i}},
\]
 where $b_{i}^{0}=\boldsymbol{1}$, $b_{i}^{k_{i}}=b_{i}^{\left(k_{i}-1\right)}b_{i}$
for $k_{i}>0$ and $b_{i}^{k_{i}}=b_{i}^{\left(k_{i}+1\right)}b_{i}^{-1}$
for $k_{i}<-1$, so $\left\{ b_{1},\ldots,b_{n}\right\} $ is a basis
for $K\left\llbracket X_{1},\ldots,X_{n}\right\rrbracket $.
\end{proof}
\begin{prop}
For every finitely generated quantity space $Q$ over $K$ there exists
an isomorphic quantity space $K\left\llbracket X_{1},\ldots,X_{n}\right\rrbracket $.
\end{prop}
\begin{proof}
Let $B=\left\{ b_{1},\ldots,b_{n}\right\} $ be a basis for $Q$ and
let $x\in Q$ have the unique expansion $x=\lambda\cdot\prod_{i=1}^{n}b_{i}^{k_{i}}$.
It is easy to verify that the mapping 
\[
\Phi:Q\rightarrow K\left\llbracket X_{1},\ldots,X_{n}\right\rrbracket ,\qquad x\mapsto\lambda X_{1}^{k_{1}}\ldots X_{n}^{k_{n}}
\]
is an isomorphism of scalable monoids between $Q$ and $K\left\llbracket X_{1},\ldots,X_{n}\right\rrbracket $.
\end{proof}
We have thus seen a simple representation of quantity spaces, in addition
to the simple axiomatic definition given earlier. While a module over
a field is a vector space, a commutative scalable monoid over a field
is unfortunately not always a quantity space, but the theory of quantity
spaces is not significantly more complicated or inaccessible than
that of vector spaces, and should belong to the basic tool-set of
mathematicians and physicists.

\end{document}